\documentclass{amsart}

\usepackage{amssymb,mathrsfs}
\usepackage{tikz}
\usepackage{xcolor}
\usetikzlibrary{matrix}
\usepackage[normalem]{ulem} 
\newtheorem{theorem}{Theorem}[section]
\newtheorem{lemma}[theorem]{Lemma}
\newtheorem{Proposition}[theorem]{Proposition}
\newtheorem{Corollary}[theorem]{Corollary}
\theoremstyle{definition}
\newtheorem{definition}[theorem]{Definition}
\newtheorem{example}[theorem]{Example}

\theoremstyle{remark}
\newtheorem{remark}[theorem]{Remark}

\numberwithin{equation}{section}

\begin{document}

\title[Translation complete subgroups]{Translation complete subgroups of affine Weyl-Heisenberg groups and their generalized wavelet systems}

\author{Narjes Rashidi}
\address{Department of Mathematics, RWTH Aachen University,  Kreuzherrenstrasse 2, 52062 Aachen, Germany}
\email{narjes.rashidi@gmail.com}

\author{Hartmut F\"uhr}
\address{Chair for Geometry and Analysis,
RWTH Aachen University, Kreuzherrenstrasse 2, 52062 Aachen, Germany}
\email{ fuehr@mathga.rwth-aachen.de}


\date{\today}



\begin{abstract}

The  $n$-dimensional affine Weyl-Heisenberg group is a Lie group typically parameterized as $G_{aWH} = \mathbb{T} \times \mathbb{R}^n \times \widehat{\mathbb{R}^n} \times \mathrm{GL}(n, \mathbb{R})$, generated by all translation, dilation, and modulation operators acting on $L^2(G)$. It was introduced by Torr\'esani and his coauthors as a common framework to discuss both wavelet and time-frequency analysis, as well as possible intermediate constructions. In this paper, we focus on a particular class of subgroups of $G_{aWH}$, namely those of the form $G = \mathbb{T} \times \mathbb{R}^n \times V \times H$, where $V$ is a subspace of $\mathbb{R}^n$ and $H$ is a closed subgroup of $\mathrm{GL}(n, \mathbb{R})$.

The main goal is to identify pairs $(V, H)$ that ensure the existence of an associated inversion formula, through the notion of square-integrable representations. We derive an admissibility criterion that is largely analogous to the well-known Calder\'on condition for the fully affine case, corresponding to $V = \{ 0 \}$. 


We then identify $G_{aWH}$ as a subgroup of the semidirect product of the $n$-dimensional Heisenberg group and the symplectic group $Sp(n,\mathbb{R})$, which acts via the extended metaplectic representation, and compare our admissibility conditions to existing criteria based on Wigner functions. 

Finally, we present a list of novel examples in dimensions two and three which illustrate the potential of our approach, and present some foundational results regarding the systematic construction, classification, and conjugacy of these groups. 

\end{abstract}

\maketitle



\section{Introduction}

The affine-Weyl Heisenberg group, denoted by $G_{aWH}$, was introduced by Torr\'esani and his coauthors in 1991 \cite{Torresani1991, Torresani1}. 
It can be understood as the group generated by the relevant unitary operators from wavelet and time-frequency analysis, namely translation operators $T_x$, dilations $D_h$, and modulations $M_\xi$, which act on $f \in L^2(\mathbb{R}^n)$ by
\[
T_x f (y) = f(y-x)~,~M_\xi f(y) = \exp(2 \pi i \langle y, \xi \rangle) f(y)~,~ D_h f(y) = |{\rm det}(h)|^{-1/2} f(h^{-1} y) ~, 
\] respectively. Throughout the paper we use $\mathcal{U}(\mathcal{H})$ to denote the unitary group of a Hilbert space $\mathcal{H}$. Furthermore, $\mathbb{R}^n$ denotes a space of column vectors, endowed with the usual scalar product and norm, and $\widehat{\mathbb{R}^n}$ its character group. For concrete computations we use the canonical identification  of $\widehat{\mathbb{R}^n}$ with $\mathbb{R}^n$ by associating to $\xi \in \mathbb{R}^n$ the character $x \mapsto \exp(2 \pi i \langle x, \xi \rangle)$; however, the distinction between $\mathbb{R}^n$ (acting via translation) and its dual group $\widehat{\mathbb{R}^n}$ (acting via modulation) provides helpful orientation in the following. 

Formally, the group is defined by
\[
  G_{aWH} = \mathbb{T} \times \mathbb{R}^n \times \widehat{\mathbb{R}^n} \times \mathrm{GL}(n, \mathbb{R})
\]
endowed with the product topology and the group action:
\begin{eqnarray} \nonumber
  \lefteqn{(z_1, x_1, \xi_1, h_1)(z_2, x_2, \xi_2, h_2)} \\ & = & (z_1 z_2 \exp(-2 \pi i \langle \xi_1, h_1 x_2 \rangle), x_1 + h_1 x_2, \xi_1 + h_1^{-T} \xi_2, h_1 h_2)\label{eqn:group_law}
\end{eqnarray}
where \( h^{-T} \) denotes \((h^{-1})^T\). This group law arises from the interaction of the associated unitary operators. This is achieved formally by considering the so-called \textit{quasi-regular representation}
$\pi : G_{aWH} \to \mathcal{U}(L^2(\mathbb{R}^n))$, which expresses the identification of the group elements with their associated operators. We therefore require that  
\[
\pi(z,0,0,I_n) = z \mathrm{id}_{L^2(\mathbb{R}^n)}, \pi (1,x,0,I_n) = T_x~,~\pi(1,0,\xi,I_n) = M_\xi~,~\pi(1,0,0,h) = D_h ~
\] Direct computation then allows us to verify that 
\begin{equation} \label{eqn:def_pi}
\pi(z,x,\xi,h) = z T_x M_\xi D_h 
\end{equation}
is indeed a strongly continuous homomorphism, i.e. a unitary representation of $G_{aWH}$. This representation is the central object of this paper.


As already observed in the initial sources, the affine Weyl-Heisenberg group provides a particularly convenient setup for the design and analysis of various \textit{hybrid} constructions of generalized wavelet and time-frequency systems that are based on suitable combinations of dilation and modulations, e.g. via the choice of certain \textit{cross-sections} (i.e. subsets) of $G_{aWH}$ that give rise to inversion formulae; see e.g. \cite{Torresani1}. 
Our paper contributes to this effort by focusing on \textit{translation complete subgroups}, i.e., subgroups $G$ of the form $G = \mathbb{T} \times \mathbb{R}^n \times V \times H$, with a subspace $V \subset \widehat{\mathbb{R}^n}$. These groups combine the full translation group (hence the terminology) with suitable choices of modulations along the subspace $V$ and dilations from the matrix group $H$. This setup covers the fully affine or generalized wavelet setting, corresponding to $V = \{0 \}$, as well as the Schr\"odinger representation underlying time-frequency analysis, which corresponds to choosing $V = \mathbb{R}^n$ and $H$ as trivial. These extreme cases have been studied quite extensively, see, e.g. \cite{Groechenig, Fuehr4}, and it is the main purpose of this paper to demonstrate that there is an easily accessible and rather large class of intermediate examples combining both nontrivial modulations and dilations, which give rise to wavelet-like expansions and inversion formulae.

\subsection{Overview and main contributions of the paper}

As stated above, we intend to provide a detailed analysis of translation complete subgroups, with the chief aim of understanding when their action via the (restriction of) the quasi-regular representation gives rise to a wavelet inversion formula based on the Haar measure of the group. More precisely, we address the following questions:
\begin{itemize}
\item Are there explicit characterizations of \textit{admissible vectors} that guarantee an inversion formula?
\item For which groups $G$ is the quasi-regular representation \textit{square-integrable}? \item For which groups $G$ is it a \textit{discrete series representation}?
\item How do our criteria relate to existing approaches studying similar questions, e.g. as im \cite{DeMari}?
\item Are there approaches to the systematic construction of such groups?
\item Are there novel examples of such groups in low dimensions?
\end{itemize}

As it turns out, the case of translation complete groups is quite similar to that of its well-understood subclass of \textit{affine groups}, i.e. those subgroups whose modulation part is trivial. The fact that all translations are included allows one to view the associated generalized wavelet transform associated to a general translation complete group as a family of convolution operators. This observation directly leads to the recognition that the reasoning developed for the affine case in \cite{Taylor, Fuehr4, Fuehr3} can be adapted to translation complete groups, essentially by replacing the linear action of the dilation group $H$ by a suitable \textit{dual action} of the semidirect product group $V \rtimes H^T$ on $\widehat{\mathbb{R}^n}$. This observation is a cornerstone of this paper and allows us to answer the questions posed above in the following way: 
\begin{itemize}
    \item The quasi-regular representation of $G$ is a discrete series representation if and only if the dual action of the semi-direct product $V \rtimes H^T$ has a unique open orbit, and in addition the stabilizer groups associated with the open orbit are compact (Corollary \ref{cor:char_discrete_series}). 
   \item There exists a Calder\'on-type admissibility condition based on integration along dual orbits (Proposition \ref{L3.1}). In the case of discrete series representations, this criterion can be expressed as the finiteness of a weighted $L^2$-norm in the Fourier domain (Corollary \ref{C3.1}) . 
    \item Weak admissibility of the group $G$, defined by the existence of a bounded \textit{injective} generalized wavelet transform $W_\psi: L^2(\mathbb{R}^n) \to L^2(G)$, is equivalent to the property that orbit space $\widehat{\mathbb{R}^n} /(V \rtimes H^T)$ is standard almost everywhere, and in addition, almost all stabilizers are compact (Theorem \ref{thm:char_weak_adm}).
    \item If $G$ is unimodular, there does not exist an \textit{isometric} wavelet transform $W_\psi: L^2(\mathbb{R}^n) \to L^2(G)$ (Corollary \ref{cor:unimod_nonadm}). In addition, the group does not admit discrete series subrepresentations (Corollary \ref{C.4.1}).
    \item $G$ is admissible if and only if it is weakly admissible and nonunimodular (Theorem \ref{thm:char_weak_adm}).
    \item It is possible to embed $G_{aWH}$ into the semidirect product $\mathbb{H}^n \rtimes Sp(n,\mathbb{R})$, and thereby identify the translation complete subgroups as a subclass of the groups studied in \cite{DeMari,Cordero,Tabacco} (see Section \ref{sect:metaplectic}). Furthermore, one can work out a direct connection between our admissibility conditions and the Wigner-function based criteria derived in \cite{DeMari}; see Proposition \ref{ref:prop_adm_cond_coincide}.
    \item Explicit examples in dimension two and three are presented in Section \ref{sect:examples}. 
\end{itemize}

Although our reasoning often follows the arguments for the affine case, as developed in \cite{Taylor, Fuehr4, Fuehr2,Fuehr3}, in most instances we have chosen to provide sufficient details to guarantee a self-contained and rigorous presentation. 




\section{Fundamentals of translation complete subgroups}

In this subsection, we establish the basic structures needed for harmonic analysis on translation complete subgroups $G = \mathbb{T} \times \mathbb{R}^n \times V \times H \subset G_{aWH}$, and the fundamental properties of $\pi|_G$. 

The first result determines which choices of $V$ and $H$ give rise to subgroups of $G_{aWH}$. It is easily verified by direct computation using (\ref{eqn:group_law}):
\begin{Proposition}
Let $V$ be a subspace of $\mathbb{R}^n$ and $H$ be a subgroup of ${\rm GL}(n, \mathbb{R})$. The set $\mathbb{T} \times \mathbb{R}^n \times V \times H$ is a subgroup of $\mathbb{T} \times \mathbb{R}^n \times \widehat{\mathbb{R}}^n \times {\rm GL}(n, \mathbb{R})$ if and only if $V$ is invariant under $H^T$, meaning that $h^T v \in V$ for all $h \in H$ and $v \in V$.
\end{Proposition}
\begin{example}
Known examples include $\mathbb{T} \times \mathbb{R}^n \times \widehat{\mathbb{R}}^n \times \{I_n\}$. Throughout this paper we will denote this group by $\mathbb{H}^n$, and call it the (reduced) \textit{Heisenberg group}. By slight abuse of notation, we will denote elements of $\mathbb{H}^n$ by $(z,x,\xi) \in \mathbb{T} \times \mathbb{R}^n \times \widehat{\mathbb{R}^n}$, i.e. we will omit the (trivial) matrix component, in settings where this notation is convenient and unambiguouos (predominantly in Section \ref{sect:metaplectic}). Another relevant subgroup is  $\{ 1 \} \times \mathbb{R}^n \times \{0\} \times \mathrm{GL}(n, \mathbb{R})$, which is isomorphic to the \textit{affine group} $\mathbb{R}^n \rtimes \mathrm{GL}(n, \mathbb{R})$.
\end{example}

An important byproduct of the proposition is the semi-direct product group $V \rtimes H^T$, which is well-defined whenever $V$ and $H$ define a translation complete subgroup. Formally, the group is defined as the set $V \times H$, endowed with the group law 
\begin{equation} \label{eqn:def_small_semi}
    (\xi_1,h_1) (\xi_2,h_2) = (\xi_1 + h_1^{-T} \xi_2, h_1 h_2)~. 
\end{equation}
Here $h^{-T}$ denotes the transpose inverse of $h$, and we identified $\xi_i \in \widehat{\mathbb{R}^n}$ with column vectors, on which the matrices act by multiplication. 
This group has a natural right action on the dual space $\widehat{\mathbb{R}^n}$, the so-called \textit{dual action} defined by
\begin{equation} \label{eqn:def_dual_action}
    \widehat{\mathbb{R}^n} \times (V \rtimes H^T) \to \widehat{\mathbb{R}^n}~,~ (\gamma, (\xi,h)) \mapsto \gamma . (\xi, h) = h^{-T}(\gamma-\xi)~,
\end{equation}
which will play a central role in the course of this paper. 

Any translation complete subgroup $G = \mathbb{T} \times \mathbb{R}^n \times V \times H$ can be written as $N \rtimes H^1$, where
\[
N = \mathbb{T} \times \mathbb{R}^n \times V \times \{I_n\},
\]
a nilpotent connected Lie group, and
\[
H^1 = \{1\} \times \{0\} \times \{0\} \times H.
\]
Clearly, $N \triangleleft G$, and $G = N \rtimes H^1$.

%

\subsection{Haar measure and modular function of translation complete subgroups}

We now determine a left Haar measure and the modular function of a general translation complete group. Note that the following results cover $G_{aWH}$ as a special case. All results are special cases of the general formulae describing Haar measures and modular functions of semidirect products, see \cite[15.29]{Hewitt}, but they can also be verified by explicit computation. 
\begin{lemma} \label{lem:Haar_meas_tc}
The Haar measure on a translation complete subgroup $\mathbb{T} \times \mathbb{R}^n \times V \times H < G_{aWH}$ is given by
\[
|\det(h^T|_V)| \, dz \, dx \, d\xi \, \frac{dh}{|\det(h)|},
\]
where $\det(h^T|_V)$ is the determinant of the map $h^T|_V: V \rightarrow V$ that sends each $v$ to $h^T v$. Here $dz,dx,d\xi$ denote the usual Lebesgue measures, normalized to one in the case of $dz$, and $dh$ denotes left Haar measure on $H$. 
\end{lemma}

It is worth realizing that typically $\det(h^T|_V)$ will be different from $\det(h^T)$. The following example makes this observation explicit. 

\begin{example} \label{ex:block_structure}
Let $V = \mathbb{R}^m \times \{0_{n-m}\}$ be a proper subspace of $\mathbb{R}^n$, $0 < m < n$, with 
$\left( \begin{array}{c} v \\ 0 \end{array} \right) \in V$. We write a general matrix $h \in \text{GL}(n,\mathbb{R})$ in block matrix form
\[
 h = \left( \begin{array}{cc} h_{11} & h_{12} \\ h_{21} & h_{22} \end{array} \right)\, 
\] with $h_{11} \in M_{m \times m}(\mathbb{R})$, and the remaining matrices are of matching sizes. We want to clarify when $h^T V \subset V$ holds.  
By block matrix calculus, this is the case precisely when, for all $v \in \mathbb{R}^{m}$,
\[
H^T \left( \begin{array}{c} v \\ 0 \end{array} \right) \in V \Longleftrightarrow \left( \begin{array}{c} h_{11}^T \cdot v \\ h_{12}^T \cdot v \end{array} \right) = \left( \begin{array}{c} v' \\ 0 \end{array} \right) 
\] and this holds if and only if $h_{12} = 0$.
Therefore,
\[
h^T = \left( \begin{array}{cc} h_{11}^T & h_{21}^T \\ 0 & h_{22}^T \end{array} \right) \quad \text{and} \quad h^T|_V = \left( \begin{array}{c} h_{11}^T \cdot v \\ 0 \end{array} \right).
\]
Obviously, $\det(h^T|_V) = \det(h_{11})$, while $\det(h) = \det(h_{11}) \det(h_{22})$.
\end{example}

We also determine the modular function and Haar integral of the semidirect product $V \rtimes H^T$. 
%
{\lemma \label{MF}The modular function of $V\rtimes H^T$ is given by
\[
 \Delta_{V\rtimes H^T}(\xi,h)=\frac{|{\rm det} (h^T|_V)|}{\Delta_H(h)},
\]
where  $\Delta_H(h)$ denotes the modular function of $H$. The left Haar measure is given by 
\[
d(\xi,h) =  |det(h^T|_V)| d\xi dh
\] with $d\xi$ denoting Lebesgue measure on the vector space $V$ and $dh$ denoting left Haar measure on $H$. 
}
%

\subsection{Generalized wavelet transforms and matrix coefficients}
Matrix coefficients are among the chief objects in abstract harmonic analysis. Given any functions $f,g \in L^2(\mathbb{R}^n)$, we let 
\[
V_f g: G_{aWH} \to \mathbb{C}~,~V_f g(\mathbf{x}) = \langle g, \pi(\mathbf{x}) f \rangle~.
\] By a slight abuse of notation, we also write $V_f g$ for the restriction of this function to any suitable subgroup. Fixing $f$ results in an operator 
\[
V_f : L^2(\mathbb{R}^n) \to C_b(G_{aWH})~,
\] where the right-hand side denotes the space of continuous and bounded functions on $G_{aWH}$. $V_f$ is called the \textbf{generalized wavelet transform} associated with $f$. By construction, this operator intertwines the action of $\pi$ with left translation on functions on $G_{aWH}$. With this notation, we can now define the central notions of this paper. 
\begin{definition} \label{def:admissible}
    Let $G < G_{aWH}$ denote a closed subgroup, and $\mathcal{H} \subset L^2(\mathbb{R}^n)$ a closed subspace that is invariant under $\pi|_G$. Let $f \in \mathcal{H}$.
    \begin{enumerate}
    \item[(a)] $f$ is called \textbf{weakly admissible for $\mathcal{H}$} if $V_f : \mathcal{H} \to L^2(G)$ is a bounded injective operator. 
    \item[(b)] $f$ is called \textbf{admissible for $\mathcal{H}$} if $V_f: \mathcal{H} \to L^2(G)$ is the nontrivial multiple of an isometry. 
    \item[(c)] $G$ is called \textbf{(weakly) admissible for $\mathbf{\mathcal{H}}$} if there exists at least one (weakly) admissible vector. In this case, the representation $\pi|_G$ is called \textbf{(weakly) square-integrable}. 
    \item[(d)] If $\pi|_G$, acting on $\mathcal{H}$, is square-integrable and irreducible, it is called a \textbf{discrete series representation}.
    \end{enumerate}
    The explicit reference to $\mathcal{H}$ will be omitted if $\mathcal{H} = L^2(\mathbb{R}^n)$.
\end{definition}
A major goal of this paper is to provide explicit criteria for the various properties.
The terminology is taken from \cite{Fuehr4}. We recall several well-known related facts and refer to \cite{Fuehr4} for proofs.
\begin{remark}
 Under the assumptions of 
    \begin{enumerate}
        \item[(a)] If $f$ is admissible, then one has the \textbf{inversion formula}
        \[
        g = \frac{1}{C_f} \int_{H} V_f g (\mathbf{x}) \pi(\mathbf{x}) f ~d\mu_G(\mathbf{x})~,
        \] holding in the weak operator sense. 
        \item[(b)] If the restriction of $\pi$ to $G$, acting on $\mathcal{H}$, is \textbf{irreducible}, then the admissibility of $0 \not= f$ is equivalent to 
        \[
        \int_G \left| V_f f(\mathbf{x}) \right|^2 d\mu_G(\mathbf{x}) < \infty~.
        \]
        \item[(c)] $\pi|_G$ is weakly square-integrable for $\mathcal{H}$ if and only if its action on $\mathcal{H}$ is unitarily equivalent to a subrepresentation of the regular representation of $G$. 
     \end{enumerate}
\end{remark}

\subsection{Invariant subspaces} 
The next step is to characterize closed subspaces that are invariant under a given translation complete subgroup $G$. Here, the advantage of translation completeness becomes apparent for the first time, since it opens the door to a systematic use of the Euclidean Plancherel transform. Also, the dual action will have its first relevant appearance in the paper. 

We start by computing the action of $\pi$ on the Plancherel transform side, obtaining the following formula for all $(z, x, \xi, h) \in G$ and $f \in \mathrm{L}^2(\mathbb{R}^n)$:
\begin{equation}\label{E.3.2}
\begin{aligned}
\widehat{(\pi(z, x, \xi, h)f)}(\gamma) &= z M_{-x} T_\xi D_{h^{-T}} \widehat{f}(\gamma) \\
&= | \det(h) |^{1/2} z e^{-2 \pi i \langle \gamma, x \rangle} \widehat{f}(h^{-T}(\gamma - \xi)),
\end{aligned}
\end{equation}
due to the equations $\widehat{T_x f} = M_{-x} \widehat{f}$, $\widehat{M_\xi f} = T_\xi \widehat{f}$, and $\widehat{D_h f} = D_{h^T} \widehat{f}$. Note that the action in the argument on the right hand side is precisely the right action of the semidirect product $V \rtimes H^T$ that was mentioned above.

The following theorem, a generalization of Theorem 2 in \cite{Hartmut}, demonstrates that any nontrivial invariant subspace corresponds to a subset $U \subset \widehat{\mathbb{R}^n}$ that is $V \rtimes H^T$-invariant.  
For its formulation, we introduce one further piece of notation: Given any Borel set,
$U \subset \widehat{\mathbb{R}^n}$ be a measurable set. Define
    \[
    \mathcal{H}_U = \{ f \in {\rm L}^2(\mathbb{R}^n) : \widehat{f}(\gamma) = 0 \text{ for almost all } \gamma \text{ outside of } U \}.
    \] Then $\mathcal{H}_U$ is a closed subspace of $L^2(\mathbb{R}^n)$.

\begin{theorem} Let $G = \mathbb{T} \times \mathbb{R}^n \times V \times H$ denote a translation complete subgroup of $G_{aWH}$.
    Let $\mathcal{H} \subset L^2(\mathbb{R}^n)$ denote a closed subspace. Then $\mathcal{H}$ is invariant under $G$ if and only if $\mathcal{H} = \mathcal{H}_U$, for some Borel subset $U \subset \mathbb{R}^n$ that is invariant under the action of $V \rtimes H^T$. 
\end{theorem}
\begin{proof}
    The proof is similar to the proof of Theorem 2 in \cite{Hartmut}. It is well-known that any fully translation-invariant subspace $\mathcal{H}$ is of the form $\mathcal{H} = \mathcal{H}_U$, for some Borel subset $U$; see e.g. \cite[9.16]{MR924157}. However, to further ensure invariance under $V \rtimes H^T$, we need in addition that $U.(\xi,h)$  and $U$ only differ by a null set, for all $(\xi,h) \in V \rtimes H^T$. Now a selection lemma due to Mackey \cite[Theorem 3]{MR143874} ensures the existence of a Borel set $U'$ that differs from $U$ only by a null set, and additionally $U'$ is $V \rtimes H^T$-invariant. This shows $\mathcal{H} = \mathcal{H}_U = \mathcal{H}_{U'}$, and $\mathcal{H}_{U'}$ is as desired. 
    
    The converse statement can be checked directly. 
\end{proof}
From now on, given any measurable $V \rtimes H^T$-invariant subset $U \subset \widehat{\mathbb{R}^n}$, we let $\pi_U$ be the restriction of $\pi$ to $\mathcal{H}_U$. Our next aim is to characterize irreducibility of $\pi_U$ by properties of $U$. 
For this purpose, 
  recall that a measure space carrying a measurable group action is called \textit{ergodic} if it cannot be decomposed into two invariant subsets of positive measure. Now the correspondence between invariant closed subspaces and invariant Borel sets established in the previous theorem immediately entails the following characterization of irreducibility. It generalizes \cite[Corollary 4]{Hartmut}.
 
\begin{Corollary} \label{cor:irred}
    Let $G = \mathbb{T} \times \mathbb{R}^n \times V \times H \subset G_{aWH}$ denote a translation complete subgroup, and let $U \subset \widehat{\mathbb{R}^n}$ denote a  Borel set of positive measure that is $V \rtimes H^T$. Then $\pi_U$ is irreducible if and only if $V \rtimes H^T$ acts ergodically on subset $U$.
\end{Corollary}

\section{Criteria for admissible vectors}

Now we set about calculating the $L^2$-norm of $V_\psi f$. As in the previous sections, we fix a translation complete subgroup $G = \mathbb{T} \times \mathbb{R}^n \times V \times H \subset G_{aWH}$. The following proposition computes the $L^2$-norm of a wavelet coefficient; it is once more an adaptation from the affine setting \cite{Taylor,Hartmut}.
\begin{Proposition}\label{L3.1}
Let $G = \mathbb{T} \times \mathbb{R}^n \times V \times H$ be a translation complete subgroup of $G_{aWH}$ and $\psi, f \in \mathrm{L}^2(G)$. Then,
\begin{equation} \label{eqn:norm_mc}
\|V_\psi f\|_{L^2(G)}^2 = \int_{\mathbb{R}^n} |\widehat{f}(\omega)|^2 \int_V \int_H |\widehat{\psi}(h^T(\omega - \xi))|^2 |{\det}(h^T|_V)| \, dh \, d\xi \, d\omega\,\,
\end{equation} in the extended sense that one side is infinite if and only if the other one is. 
Define $\Phi_\psi: \mathbb{R}^n \rightarrow \mathbb{R}^+$ by
\[ 
\Phi_\psi(\omega) := \int_V \int_H |\widehat{\psi}(h^T(\omega - \xi))|^2 \, d(h, \xi),
\]
where $d(h, \xi) = |{\det}(h^T|_V)| \, dh \, d\xi$. Then,
\begin{equation} \label{eqn:norm_mc2}
\|V_\psi f\|_{L^2(G)}^2 = \int_{\mathbb{R}^n} |\widehat{f}(\omega)|^2 \Phi_\psi(\omega) \, d\omega,
\end{equation} again including the case that both sides are infinite. 
\end{Proposition}
 
\begin{proof}
We use the Haar measure determined in Lemma \ref{lem:Haar_meas_tc} to compute for every $\psi, f\in {\rm L}^2(\mathbb{R}^n)$ \\
\begin{equation*}
\begin{aligned} 
 ||V_\psi f||_{L^2(G)}^2 &=&& \int_G |\langle f, \pi(z,x,\xi,h)\psi\rangle|^2  d\mu_G(z,x,\xi,h)\\
 &=&&\int_{\mathbb{R}^n} \int_{V} \int_{H} |\langle f, T_xM_\xi D_h \psi\rangle|^2 |{\rm det}(h^T|_V)|\frac{dh}{|{\rm det}(h)|}d\xi dx \\
 &=&& \int_{V} \int_{H}\int_{\mathbb{R}^n}|\langle \widehat f, \widehat{T_xM_\xi D_h \psi}\rangle|^2 dx |{\rm det}(h^T|_V)|\frac{dh}{|{\rm det}(h)|}d\xi\\
 &=&&\int_{V} \int_{H}\int_{{\mathbb{R}^n}} |\langle \widehat f,M_{-x} T_\xi D_{h^{-T}}\widehat{\psi}\rangle|^2 dx |{\rm det}(h^T|_V)|\frac{dh}{|{\rm det}(h)|}d\xi\\
  &=&&\int_{V} \int_{H}\int_{{\mathbb{R}^n}} |\int_{{\mathbb{R}^n}} \widehat f(\omega) e^{-2\pi i \omega x} T_\xi D_{h^{-T}}\widehat{\overline{\psi}}(\omega)d\omega|^2 dx |{\rm det}(h^T|_V)|\frac{dh}{|{\rm det}(h)|}d\xi\\
 &=&&\int_{V} \int_{H}\int_{{\mathbb{R}^n}} | \mathcal{F}(\widehat f\cdot T_\xi D_{h^{-T}}\widehat{\overline{\psi}})(x)|^2 dx |{\rm det}(h^T|_V)|\frac{dh}{|{\rm det}(h)|}d\xi\\
 &=&& \int_{{\mathbb{R}^n}} |\widehat f(x)|^2 \int_V \int_H |T_\xi D_{h^{-T}}\widehat{\psi}(x)|^2 |{\rm det}(h^T|_V)|\frac{dh}{|{\rm det}(h)|}d\xi dx\\
  &=&& \int_{{\mathbb{R}^n}} |\widehat f(x)|^2 \int_V \int_H |{\rm det}(h)||T_\xi \widehat{\psi}(h^T x)|^2 |{\rm det}(h^T|_V)|\frac{dh}{|{\rm det}(h)|}d\xi dx\\
   &=&& \int_{{\mathbb{R}^n}} |\widehat f(x)|^2 \int_V \int_H | \widehat{\psi}(h^T (x-\xi))|^2 |{\rm det}(h^T|_V)|dhd\xi dx\\
   &=&&\int_{\mathbb{R}^n} |\widehat{f}(\omega)|^2 \Phi_\psi(\omega)d\omega,
\end{aligned}
\end{equation*}
where $\mathcal{F}$ denotes the Fourier transform.  The computation uses several well-known properties of the Fourier transform, such as its unitarity and its interplay with the representation $\pi$ (see equation (\ref{E.3.2})), but also the normalization of Haar measure on $\mathbb{T}$. This proves (\ref{eqn:norm_mc}), and (\ref{eqn:norm_mc2}) is just a reformulation.  
\end{proof}
Now straightforward measure-theoretic arguments allow to derive the following characterization of (weakly) admissible vectors. 
\begin{Corollary}\label{C3.1}
Let $G = \mathbb{T} \times \mathbb{R}^n \times V \times H$ be a subgroup of $G_{aWH}$, and assume that $\mathcal{H} = \mathcal{H}_U$ is a $G$-invariant closed subspace of $L^2(\mathbb{R}^n)$. Then $\psi \in \mathcal{H}_U$ is \textbf{weakly admissible} if and only if there is a constant $C_\psi>$ such that the auxiliary function 
\begin{equation}\label{eqn:Phipsi}
 \Phi_\psi(\omega) = \int_V \int_H |\widehat{\psi}(h^T(\omega-\xi))|^2 |{\det}(h^T|_V)| \, dh \, d\xi
\end{equation}
fulfills 
\[
0 < \Phi_\psi(\omega) < C_\psi~,\mbox{a.e. } \omega \in U
\] Furthermore, $\psi$ is \textbf{admissible} if and only if 
\begin{equation}\label{E3.1}
 \Phi_\psi(\omega) =  C_\psi~, \mbox{a.e. } \omega \in U,
\end{equation}
for a constant $C_\psi>0$.
\end{Corollary}
\begin{proof}
By the lemma and definition above, for all \( f \in \mathrm{L}^2(\mathbb{R}^n) \) we have
\[
 \|V_\psi f\|_{L^2(G)}^2 = c \|f\|_{L^2(\mathbb{R}^n)}^2 \Longleftrightarrow \Phi_\psi \stackrel{\text{a.e.}}{=} c.
\] This establishes the admissibility criterion (\ref{E3.1}). The condition for weak admissibility is proved analogously. 
\end{proof}

We note an important consequence of the required finiteness of $\Phi_\psi$, which follows from \cite[Lemma 11]{Hartmut}. 
\begin{Corollary} \label{cor:char_adm}
Let $G = \mathbb{T} \times \mathbb{R}^n \times V \times H < G_{aWH}$ be a translation complete subgroup and $\mathcal{H} = \mathcal{H}_U$ a $G$-invariant closed subspace. Assume that $G$ is weakly admissible for $\mathcal{H}_U$. Then, for almost all $\omega \in U$, the associated stabilizer
\[
(V \rtimes H^T)_{\omega} = \{ (\xi,h) \in V \rtimes H^T ~:~ h^T(\omega - \xi) = \omega \}
\] is compact. 
\end{Corollary}





At this stage we can already briefly comment on the extreme case $V = \widehat{\mathbb{R}^n}$. In this case $G$ contains the reduced Heisenberg group, and the restriction of $\pi$ to this subgroup coincides with the Schr\"odinger representation, which is known to be a discrete series representation.  Hence we can derive that $V = \widehat{\mathbb{R}^n}$ always entails irreducibility of $\pi$. Furthermore, direct calculation allows to verify that the discrete series property is preserved whenever $H$ is compact. This compactness property also turns out to be necessary for the discrete series property, by Corollary \ref{cor:char_adm}. See Example \ref{ex:main_examples} below for a more complete argument. 



As another extreme case, we consider the choice $V = \{ 0 \}$, which leads to considering semidirect product groups $G = \mathbb{R}^n \rtimes H$ and their associated wavelet systems. This setting is well-studied, beginning with \cite{Taylor}. We refer in particular to \cite{Fuehr4} for a systematic treatment of admissibility conditions beyond the discrete series case. We will repeatedly refer to this class of groups as \textit{the affine case}. 


\section{Criteria for (weakly) admissible subgroups}
\label{sect:adm_criteria}

In this section we characterize reproducing translation complete subgroups, using a detailed measure-theoretic analysis of the dual action. We therefore fix a translation complete subgroup $G = \mathbb{T} \times \mathbb{R}^n \times V \times H$ and a $V \rtimes H^T$-invariant Borel subset $U \subset \widehat{\mathbb{R}^n}$. We want to characterize when there exist (weakly) admissible vectors for $\mathcal{H}_U$. Once again, the affine case provides a valuable template for our arguments. The main references for this section are \cite{Fuehr4,Fuehr3}.

A rigourous formulation requires various notions from ergodic theory, which we now recall. 
A Borel space is called {\bf countably generated} if
the $\sigma$-algebra is generated by a countable subset. 
It is called {\bf separated} if single points are Borel.
A Borel space is called {\bf countably separated} if there is a sequence of
Borel sets separating the points. All these properties are
inherited by products and Borel subspaces.
A Borel space is called {\bf analytic} if it is (Borel-isomorphic to) the
Borel image of a standard space in a countably generated space.

The orbit space $U/(V \rtimes H^T)$ is endowed with the quotient Borel structure, i.e. $A \subset U/(V \rtimes H^T)$ is Borel if and only if $\bigcup_{\mathcal{O} \in A} \mathcal{O} \subset U$ is Borel. $\lambda$ denotes the Haar measure on $\widehat{\mathbb{R}^n}$, which is just the usual Lebesgue measure. 
We say that the orbit space $U/(V \rtimes H^T)$ admits a $\lambda$-transversal if there exists an invariant
$\lambda$-conull Borel set $Y \subset U$ and a Borel set $C\subset Y$ meeting each orbit 
in $Y$ in precisely one point.

A {\bf pseudo-image} of $\lambda$ is a measure $\overline{\lambda}$
on $U/(V \rtimes H^T)$ obtained as image measure of an equivalent finite measure under the quotient map;
clearly all pseudo-images are equivalent.  We call
$\overline{\lambda}$ {\bf standard}, if there exists $Y \subset U$ Borel,
invariant and conull, such that $Y/(V \rtimes H^T)$ is standard.

Finally, we need the notion of a {\bf measure decomposition}: 
A {\bf measurable family of measures} is a family $(\beta_{\mathcal{O}})_{\mathcal{O}
\subset X}$ indexed by the orbits in $X$, such that for all Borel sets $B \subset X$,
the map $\mathcal{O} \mapsto \beta_{\mathcal{O}}(B)$ is Borel on $X/H$. 

A  {\bf measure decomposition of $\mathbf{\lambda}$} consists of a pair 
$(\overline{\lambda},(\beta_{\mathcal{O}})_{\mathcal{O}
\subset X})$ , where  $\overline{\lambda}$
is a pseudo-image of $\lambda$ on $X/H$, or a $\sigma$-finite measure equivalent
to such a pseudo-image, and a measurable family $(\beta_{\mathcal{O}})_{\mathcal{O}
\subset X}$ such that for all $B \subset X$ Borel,
\[
 \lambda(B) = \int_{X/H} \beta_{\mathcal{O}}(B) d\overline{\lambda}(\mathcal{O})~.
\]
Note that this entails, for all positive Borel functions $f$ on $X$, that
\[
 \int_{X} f(\xi) d\lambda(\xi) = \int_{X/H} \int_{\mathcal{O}} f(\xi) d\beta_{\mathcal{O}}(\xi) d\overline{\lambda}(\mathcal{O}). 
\]
We say that {\bf $\mathbf{\lambda}$ decomposes over the orbits} if there exists a measure decomposition
with the additional requirement that, for $\overline{\lambda}$-almost every $\mathcal{O} \in  X/H$, 
the measure $\beta_{\mathcal{O}}$ is supported in $\mathcal{O}$, meaning 
$\beta_{\mathcal{O}}(X\setminus \mathcal{O}) = 0$. 

With this terminology, we can now formulate the pivotal result of this section: 
\begin{theorem} \label{thm:char_weak_adm}
Let $G = \mathbb{T} \times \mathbb{R}^n \times V \times H$ denote a translation complete subgroup, and $U \subset \widehat{\mathbb{R}^n}$ a $V \rtimes H^T$-invariant Borel set of positive measure. Let 
\[
U_c = \{ \xi \in U ~:~ (V \rtimes H^T)_\xi \mbox{ is compact } \}~.
\]
Then $U_c$ is a $V \rtimes H$-invariant Borel set, and the following are equivalent: 
\begin{enumerate}
\item[(i)] $G$ is weakly admissible for $\mathcal{H}_U$.
\item[(ii)] $\lambda (U \setminus \lambda_{U_c}) = 0$, and $\lambda$ decomposes over the orbit space $U/(V \rtimes H^T)$.
\item[(iii)] $\lambda (U\setminus \lambda_{U_c}) = 0$, and $\overline{\lambda}$ is standard on $U$.
\item[(iv)] $\lambda (U\setminus \lambda_{U_c}) = 0$, and $U/(V \rtimes H^T)$ admits a $\lambda$-transversal.
\end{enumerate}
\end{theorem}
\begin{proof}
The wavelet version of this result was proved in \cite{Fuehr3}, and our argument is a fairly straightforward adaptation. We therefore only sketch it. 

Recall that by Corollary \ref{cor:char_adm}, $G$ is admissible for $U$ if and only if there exists $\psi \in \mathcal{H}_U$ satisfying
\begin{equation}
 0 < \int_V \int_H |\widehat{\psi}(h^T(\omega-\xi))|^2 |{\det}(h^T|_V)| \, dh \, d\xi < C_\psi~~,\mbox{a.e. }\omega  \in U~.
\end{equation}
This condition is easily seen (via \cite[Lemma 8]{Fuehr3}) to be equivalent to the property that the dual action of $V \rtimes H^T$ is weakly admissible in the sense of \cite[Definition 7]{Fuehr3}. 

Now first assume that (i) holds. Then the set $U_c \subset U$ is Borel and conull by \cite[Lemma 11]{Fuehr3}. Furthermore, \cite[Theorem 12]{Fuehr3} applies to yield the desired equivalence to (ii)--(iv), as well as the converse direction. 
\end{proof}

The theorem is useful in several ways. Most importantly, it is a sharp characterization of weakly admissible groups. Secondly, as we shall see in the next subsection, it provides easy access to a characterization of the discrete series case. Finally, the measure decomposition will turn out to be the key to the characterization of admissible groups. 

\subsection{Characterizing discrete series subrepresentations}
We start with a characterization of the discrete series case. Once again, the wavelet case provides the blueprint. 
\begin{Corollary} \label{cor:char_discrete_series} Let $G = \mathbb{T} \times \mathbb{R}^n \times V \times H$ denote a translation complete subgroup, and $U \subset \widehat{\mathbb{R}^n}$ a $V \rtimes H^T$-invariant Borel set of positive measure. 
Then the quasi-regular representation, acting on $\mathcal{H}_U$, is a discrete series representation if and only if $U$ contains a conull open orbit $\mathcal{O} \subset U$ of positive measure, with the additional property that the stabilizers associated with the open orbit are compact. 
\end{Corollary}
\begin{proof}
If $\pi_U$ is a discrete series representation, then $G$ is weakly admissible for $\mathcal{H}_U$. Hence $\overline{\lambda}$ is standard on $U/(V \rtimes H^T)$. On the other hand, irreducibility of the representation entails that the action on $U$ is ergodic, by Corollary \ref{cor:irred}. Now the standardness of the orbit space entails, by \cite[Proposition I.3.9]{Moore}, the existence of a conull orbit $\mathcal{O} \subset U$, and the stabilizers of that orbit must be compact by Theorem \ref{thm:char_weak_adm}. 

Conversely, if an open conull orbit $\mathcal{O} \subset U$ exists, then $\mathcal{H}_U = \mathcal{H}_{\mathcal{O}}$ is weakly square integrable and irreducible, hence a discrete series representation. 

Hence, it remains only to show that orbits of positive measure are already open. This is provided by the next lemma.
\end{proof}

The next results are rather straightforward adaptations from the fully affine case (see e.g. \cite{Fuehr2}), which is why we omit the proofs. 

{\lemma  \label{L3.2} Let $G$ be a (second countable) Lie group acting smoothly on the ${C}^{\infty}$-manifold
$X$, and let $\mathcal{O}$ be an orbit in $X$ with positive measure (with respect to the Lebesgue measure class on $X$). Then
$\mathcal{O}$ is open. }


Let us denote the Lie algebra of $V\rtimes H$  by $\mathfrak{v}\rtimes \mathfrak{h}$. Then an adaptation of  Lemma 2 from \cite{Fuehr2} to $V\rtimes H$ gives the following rank condition for open orbits. 

{\lemma \label{L1} Let $V\rtimes H^T$ be a closed subgroup with Lie algebra $\mathfrak{v}\rtimes \mathfrak{h}^T$.
For each $\gamma\in \mathbb{R}^n$, the orbit $\mathcal{O}_\gamma=\gamma. (V \rtimes H^T) $ is open if and only if the differential of the map
$(\xi,h)\mapsto h^{T}(\gamma-\xi)$ has full rank at the neutral element. More explicitly, $\mathcal{O}_\gamma$ is open if and only if the map
$dp_\gamma: \mathfrak{v}\rtimes \mathfrak{h}^T\rightarrow \mathbb{R}^n$ is onto.
}

 The following is a straightforward adaptation of Theorem 3 from \cite{Fuehr2} to the action $V\rtimes H$. 
\begin{theorem} \label{thm:dec_open_orbits}
   Let $V\rtimes H^T$ be closed. If there exists $x\in \widehat{\mathbb{R}^n}$ such that $\mathcal{O}_x$ is open, then,
up to null sets, $\widehat{\mathbb{R}^n}$ is the union of finitely many open $V \rtimes H^T$-orbits. 
\end{theorem}

%

For discrete series representations, it is possible to replace the slightly unwieldy integral over $V \rtimes H^T$ in the admissibility condition from Corollary \ref{C3.1} by an integral against a suitable weight in the Fourier domain. 
The wavelet counterpart of the following result is \cite[Theorem 13]{Hartmut}, with a largely analogous proof, hence we omit the proof here. 
\begin{theorem}
\label{T3.1} 
Assume that $G = \mathbb{T} \times \mathbb{R}^n \times V \times H$ is a translation complete subgroup, and that $U \subset \widehat{\mathbb{R}^n}$ is an open orbit with compact dual stabilizers. Fix $\gamma \in U$. The function $\Psi$ defined on $U$ by 
\begin{equation*}
\Psi(h^T(\gamma-\xi)) := \frac{\Delta_{V \rtimes H^T}(\xi,h)}{|{\rm det}(h)|}
\end{equation*}
is well-defined and continuous on $U$, and for every $\psi \in \mathcal{H}_U$, $\psi$ is admissible if and only if 
\[ \int_U |\widehat{\psi}(\eta)|^2 \Psi(\eta) \, d\lambda(\eta) < \infty.
\]
\end{theorem}

Note that just as in the affine case described in \cite{Taylor,Fuehr4}, the weight function $\Psi$ can be used for an explicit description of the \textbf{Duflo-Moore operator} associated with each discrete series representation. 
\subsection{Characterization of admissible groups}

We now turn to the characterization of admissible groups. We start with the assumption that $G = \mathbb{T} \times \mathbb{R}^n \times V \times H$ is weakly admissible, i.e. there exists $\psi_1 \in \mathcal{H}_U$ such that 
\[ 0 < \int_V \int_H |\widehat{\psi_1}(h^T(\omega-\xi))|^2 |{\det}(h^T|_V)| \, dh \, d\xi < C_\psi~~,\mbox{a.e. }\omega  \in U~,\]
and we want to know when there exists $\psi_2 $
\[  \int_V \int_H |\widehat{\psi_2}(h^T(\omega-\xi))|^2 |{\det}(h^T|_V)| \, dh \, d\xi =1 ~~,\mbox{a.e. }\omega  \in U~.\]
As it turns out, this involves measure-theoretic conditions on $U$. In particular,   the measure decomposition on $U$, as provided by Theorem \ref{thm:char_weak_adm}, will become particularly important. Another relevant issue in the following will be whether $G$ is unimodular or not. 

Hence, fix any pseudo-image 
 $\overline{\lambda}$ of $\lambda$ on the orbit space 
 $U/(V \rtimes H^T) \simeq \{\mathcal{O}_\gamma\}_{\gamma\in U}$. As noted above, 
there exists  a family of measures $(\beta_{\mathcal{O}_\gamma})_{\mathcal{O}}$ such that $ d\lambda(x)=d\beta_{\mathcal{O}_\gamma}(x)d\overline{\lambda}(\mathcal{O}_\gamma)$, and each $\beta_{\mathcal{O}}$ is supported on $\mathcal{O}$. Furthermore, the measures are essentially uniquely determined: If $\overline{\lambda}_1$ and $\left( \tilde{\beta}_{\mathcal{O}} \right)_{\mathcal{O} \in U/(V \rtimes H^T)}$ is a second measure decomposition, then $\tilde{\beta}_{\mathcal{O}} = \alpha(\mathcal{O}) \beta_{\mathcal{O}}$, and $\frac{d\overline{\lambda}_1}{d\overline{\lambda}}(\mathcal{O}) = \alpha(\mathcal{O})$, where $\alpha: U/(V \rtimes H^T) \to \mathbb{R}^+$ is a suitable measurable function. 

On the other hand, the fact that all stabilizers associated with $U$ are compact allows us to introduce a measure $\mu_{\mathcal{O}}$ on the orbit $\mathcal{O}$, by picking $\gamma \in \mathcal{O}$, and letting 
\[ \mu_{\mathcal{O}}(A) = \mu_{V \rtimes H^T}( \{ (\xi,h) \in V \rtimes H : \gamma. (\xi,h) \in A \} )~.
\]
Just as $\beta_{\mathcal{O}}$, the measure $\mu_{\mathcal{O}}$ is $\sigma$-finite and quasi-invariant, and therefore $\beta_{\mathcal{O}}$ and $\mu_{\mathcal{O}}$ are equivalent. Hence, there exists an essentially unique Borel function $\kappa:U\rightarrow \mathbb{R}^+_0$ such that for 
$\overline{\lambda}$-almost all orbits,
 \[
 \kappa(x)=\frac{d\lambda}{d\mu}(x)=\frac{d\beta_{\mathcal{O_\gamma}}}{d\mu_{\mathcal{O_\gamma}}}(x).
 \]
 
So, we can write 
\begin{equation}\label{eq:a2}
   \int_U
 f(x) d\lambda(x)=\int_{\mathcal{O}}\int_{\mathcal{O}_\gamma} f(x) d\mu_{\mathcal{O}_\gamma}(x)
\kappa(x)d\overline{\lambda}(\mathcal{O}_\gamma),
\end{equation}
where $\kappa(x)$ is a global Radon-Nikodym derivative. In particular, $\kappa$ is constant on the orbits (or $\kappa$ is $V\rtimes H^T$-invariant) if and only if $G$ is unimodular.

Assuming now that $G$ is unimodular, we can simultaneously regard $\kappa$ as a measurable map on the orbit 
space and rewrite 
$$d\beta_{\mathcal{O}_\gamma}(x)=\frac{d\lambda}{d\mu}(x)d\mu_{\mathcal{O}_\gamma}(x)=
\overline\kappa(\mathcal{O}_\gamma)d\mu_{\mathcal{O}_\gamma}(x),$$ 
leading to 
\begin{equation} \label{eqn:meas_decomp_unimod}
   d\lambda(x)=\kappa(\mathcal{O}_\gamma)d\mu_{\mathcal{O}_\gamma}(x)d\overline{\lambda}(\mathcal{O}_\gamma) = d\mu_{\mathcal{O}_\gamma}(x) d\overline{\lambda}_{\kappa}(\mathcal{O}_\gamma)~,
   \end{equation}
 using the new measure $\overline{\lambda}_{\kappa}$
  \[
  d\overline{\lambda}_{\kappa}(\mathcal{O}_\gamma) = \kappa(\mathcal{O}) d\overline{\lambda}(\mathcal{O})
  \]
 on the orbit space. 

 The following theorem demonstrates the usefulness of this particular normalization: 
{\theorem \label{th.3} Assume that $G=\mathbb{T}\times \mathbb{R}^n\times V \rtimes H^T$ is a unimodular translation complete subgroup of $G_{aWH}$. Let $U \subset \widehat{\mathbb{R}^n}$ be a $V \rtimes H^T$-invariant Borel set of positive measure.
Then $G$ is admissible for $\mathcal{H}_U$ if and only if $G$ is weakly admissible for $\mathcal{H}_U$, and in addition
\[
\overline{\lambda}_{\kappa}(U/(V \rtimes H^T))<\infty,
\]
where $\overline{\lambda}_{\kappa}$ is defined via (\ref{eqn:meas_decomp_unimod}).
}
\begin{proof} Assume $G$ is unimodular and there exists an admissible vector $\psi\in \mathcal{H}_U$, then
\begin{equation*}
\begin{aligned}
 ||\psi||^2_2 &=&& \int_{U} |\widehat{\psi}(x)|^2d\lambda(x)\\
 &=&& \int_{\mathcal{O}}\int_{\mathcal{O}_\gamma}|\widehat{\psi}(x)|^2 d\mu_{\mathcal{O}_\gamma}(x)d\overline{\lambda}_{\overline{\kappa}}(\mathcal{O}_\gamma)\\
 &=&& \int_{\mathcal{O}}\int_{V \rtimes H^T}|\widehat{\psi}(\gamma\cdot(\xi,h))|^2 d\mu_{V \rtimes H^T}(\xi,h)d\overline{\lambda}_{\overline{\kappa}}(\mathcal{O}_\gamma)\\
 &=&& \int_{\mathcal{O}} 1 d\overline{\lambda}_{\overline{\kappa}}(\mathcal{O}_\gamma)\\
 &=&& \overline{\lambda}_{\overline{\kappa}} (\mathcal{O}),
\end{aligned}
\end{equation*}
and due to the admissibility of $\psi$ we can conclude $ \overline{\lambda}_{\overline{\kappa}} (\mathcal{O})<\infty$.\\
For the converse, assume that there exists a vector $\psi_0 \in \mathcal{H}_U$ satisfying 
\begin{equation*}
 0<\int_{V \rtimes H^T} |\widehat{\psi}_0(\gamma\cdot (\xi,h))|^2 d\mu(\xi,h)<\infty  \quad ~~~ \text{for}~~ \lambda-a.e~~~ \gamma\in U. 
\end{equation*}

Now, define
\begin{equation*}
\Phi(\gamma)= \left( \int_{V \rtimes H }|\widehat{\psi_0}(\gamma\cdot(\xi,h))|^2 d\mu_{V \rtimes H }(\xi,h)\right)^{1/2}.
\end{equation*}
Let $\widehat{\psi}(\gamma)= \widehat{\psi_0}(\gamma)/\Phi(\gamma)$. A simple calculation shows that $\widehat{\psi}$ satisfies the Calderon condition:
\begin{equation*}
 \begin{aligned}
  \int_{V \rtimes H }|\widehat{\psi}(\gamma\cdot(\xi,h))|^2 d\mu_{V \rtimes H }(\xi,h)
  &=&&\int_{V \rtimes H^T} |\widehat{\psi_0}(\gamma\cdot(\xi,h))|^2d\mu_{V \rtimes H }(\xi,h)\cdot \frac{1}{\Phi(\gamma)}\\ 
  &=&& 1 \hspace{0.5cm} (\gamma-a.e).
 \end{aligned}
\end{equation*}
Since $G$ is unimodular, the measure decomposition allows us to compute 
\begin{equation*}
\begin{aligned}
 ||\psi||^2_2 
 &=&& \int_{\mathcal{O}}\int_{\mathcal{O}_\gamma}|\psi(x)|^2 d\mu_{\mathcal{O}_\gamma}(x) d\overline{\lambda}_{\overline\kappa}(\mathcal{O}_\gamma)\\
 &=&& \overline{\lambda}_{\overline\kappa} (\mathcal{O}).
\end{aligned}
\end{equation*}
 By assumption $\overline{\lambda}_{\overline\kappa} (\mathcal{O})<\infty$, thus the inverse Plancherel transform of $\widehat{\psi}$ is well-defined element of $\mathcal{H}_U$, and by construction it is admissible
for $\mathcal{H}_U$.
\end{proof}

The following Corollary is a direct consequence of the last paragraph of the proof: 
{\Corollary Assume that $G$ and $U$ are as in Theorem \ref{th.3}. Then the norm of an admissible vector $\psi \in \mathcal{H}_U$ is independent of the choice of $\psi$, i.e., if $\psi_1$ and $\psi_2$ are two such 
admissible vectors, then $||\psi_1||=||\psi_2||$.}

We continue our study of the unimodular case, mostly with the aim of showing a number of no-go results.
\begin{Proposition}\label{P.4.1} Assume that $G=\mathbb{T}\times \mathbb{R}^n\times V \rtimes H^T$ is a unimodular translation complete subgroup of $G_{aWH}$. Let $U \subset \widehat{\mathbb{R}^n}$ be an $V \rtimes H^T$-invariant Borel set of positive measures with the property that $G$ is weakly admissible for $\mathcal{H}_U$. Further assume that $\mathbb{R}^+ U = U$.
 Let the pair $(\overline{\lambda}_\kappa,\beta_{\mathcal{O}_\gamma})$ be a measure decomposition of Lebesgue measure on $U \subset \widehat{\mathbb{R}^n}$, thus
\[
d\lambda(x)=d\beta_{\mathcal{O}_\gamma}(x)d{\overline\lambda}_V({\mathcal{O}_\gamma}).
\]
For $a\in \mathbb{R}^*$ and $\gamma\in U$, let $a^*(\beta_{\mathcal{O}_\gamma})$
 denote the image measure of $\beta_{\mathcal{O}_\gamma}$ on $\mathcal{O}_{a\gamma}$, i.e., for a measurable set 
 $A\subset \mathcal{O}_{a\gamma}$, 
$a^*(\beta_{\mathcal{O}_\gamma})(A):=\beta_{\mathcal{O}_\gamma}(a^{-1}A)$. Moreover, let the measure $\overline{\lambda}_a$ be given by
$\overline{\lambda}_a(B):=\overline{\lambda}_V(aB)$, where $B$ is a subset of the orbit space. Then the following relations hold:
\begin{equation*}
 \begin{aligned}
  \beta_{\mathcal{O}_{a\gamma}}&=&&|a|^{\dim V} a^*(\beta_{\mathcal{O}_\gamma})\\
   \overline{\lambda}_a &=&&|a|^{n-\dim V}\overline{\lambda}_V\\
 \end{aligned}
\end{equation*}
\end{Proposition}
\begin{proof} The first equation is a straightforward result of the definition of $\beta_{\mathcal{O}}$.\\
(i) First notice for $a\in \mathbb{R}^*$, $a\mathcal{O}_\gamma=\mathcal{O}_{a\gamma}$:
\begin{equation*}
 \begin{aligned}
  a\mathcal{O}_\gamma&=& a\{ \gamma\cdot(\xi,h): (\xi,h)\in V \rtimes H^T\}, \\
 \mathcal{O}_{a\gamma}&=&\{ a\gamma\cdot(\xi,h): (\xi,h)\in V \rtimes H^T\}.
 \end{aligned}
 \end{equation*}
(ii) Let $A\subset \mathcal{O}_\gamma$, then 

\begin{equation*}
 \begin{aligned}
 \beta_{\mathcal{O}_{a^{-1}\gamma}}(A) &=&&\int_{\mathcal{O}_{a^{-1}\gamma}} \chi_A(x) \beta_{\mathcal{O}_{a^{-1}\gamma}}(x)\\
 &=&&\int_{V \rtimes H^T} \chi_A(a^{-1}\gamma\cdot(\xi,h))d\mu_{V \rtimes H^T}(\xi,h) \\
 &=&&\int_{V \rtimes H^T}\chi_{aA}(\gamma\cdot(a\xi,h)) d\mu_{V \rtimes H^T}(\xi,h)\\
  &=&&|a|^{-\dim V}\int_{V \rtimes H^T}\chi_{aA}(\gamma\cdot(\xi,h)) d\mu_{V \rtimes H^T}(\xi,h)\\
 &=&&|a|^{-\dim V}\int_{\mathcal{O}_{\gamma}}\chi_{aA}(x) d\beta_{\mathcal{O}_{\gamma}}(x)\\
 &=&&|a|^{-\dim V}\beta_{\mathcal{O}_{\gamma}}(aA),
 \end{aligned}
\end{equation*}
using that $|a|^{\dim V}$ equals the determinant of $a \cdot \text{id}_{V}$.
 
To show  $\overline{\lambda}_V(aB)=|a|^{n-\dim V}\overline{\lambda}_V(B)$, assume $f: U\rightarrow \mathbb{R}$ is a positive measurable function. Let $I_\mathcal{O}(f)$ denote the function on orbit space defined by
\[
 I_\mathcal{O}(f)(\mathcal{O}_\gamma):=\int_{\mathcal{O}_\gamma} f(x) d\beta_{\mathcal{O}_\gamma}(x).
\]
Moreover, let $f_a(x):=f(a^{-1}x)$, for all $x\in U\approx \mathbb{R}^n$ and $a\in \mathbb{R}^*$. The above computation implies
\begin{equation*}
 \begin{aligned}
  I_\mathcal{O}(f_a)(\mathcal{O}_\gamma)&=&& \int_{\mathcal{O}_\gamma} f(a^{-1}x) d\beta_{\mathcal{O}_\gamma}(x)\\
  &=&&|a|^{-\dim V} \int_{\mathcal{O}_{a^\gamma}} f(x) d\beta_{\mathcal{O}_{a\gamma}}(x)\\
  &=&& |a|^{-\dim V} I_\mathcal{O}(f)(\mathcal{O}_{a\gamma}).
 \end{aligned}
\end{equation*}

Now we compute the integral of $I_\mathcal{O}(f)$ on $\frac{U}{V \rtimes H^T}\simeq \mathcal{O}$:
\begin{equation*}
\begin{aligned}
 \int_{\mathcal{O}} I_\mathcal{O}(f)(\mathcal{O}_\gamma)d\overline{\lambda}_V(\mathcal{O}_\gamma)
 &=&& |a|^{n} \int_U f_a(x) d\lambda(x)\\
 &=&& |a|^{n-\dim V}|a|^{\dim V}\int_\mathcal{O} \int_{\mathcal{O}_\gamma} f_a(x) d\beta_{\mathcal{O}_\gamma}(x)d\overline{\lambda}_V(\mathcal{O}_\gamma)\\
 &=&& |a|^{n-\dim V}\int_\mathcal{O} |a|^{\dim V} I_\mathcal{O}(f_a)(\mathcal{O}_\gamma) d\overline{\lambda}_V(\mathcal{O}_\gamma)\\
 &=&& |a|^{n-\dim V}\int_\mathcal{O} I_\mathcal{O}(f)(\mathcal{O}_{a\gamma}) d\overline{\lambda}_V(\mathcal{O}_\gamma)\\
 &=&& |a|^{n-\dim V}\int_\mathcal{O} I_\mathcal{O}(f)(\mathcal{O}_{\gamma}) d\overline{\lambda}_a(\mathcal{O}_\gamma).\\
\end{aligned}
\end{equation*}
On the other hand,
\begin{equation*}
\begin{aligned}
 \int_{\mathcal{O}} I_\mathcal{O}(f)(\mathcal{O}_\gamma)d\overline{\lambda}_V(\mathcal{O}_\gamma)&=&&\int_U f(x) d\lambda(x)\\
 &=&& |a|^{n}\int_U f_a(x) d\lambda(x).
 \end{aligned}
\end{equation*}
\end{proof}
{\Corollary \label{cor:scaling_invariant} Assume that $G=\mathbb{T}\times \mathbb{R}^n\times V \times H$ is a unimodular translation complete subgroup of $G_{aWH}$, with $V \subsetneq \widehat{\mathbb{R}^n}$. Let $U \subset \widehat{\mathbb{R}^n}$ be a $V \rtimes H^T$-invariant Borel set of positive measure, satisfying $U=\mathbb{R}^+U$. Then $G$ is not admissible for $\mathcal{H}_U$.}

\begin{proof} There is nothing to show if $G$ is not even weakly admissible for $\mathcal{H}_U$. Assuming that $G$ is weakly admissible for $\mathcal{H}_U$, we make use of the measure decomposition described at the beginning of this subsection. By assumption, $\overline{\lambda}_{\kappa}(U/ (V \rtimes H^T)) >0$. Using scaling invariance of $U$ and Proposition \ref{P.4.1}, this entails
\[
\overline{\lambda}_{\kappa}(U/ (V \rtimes H^T)) = \overline{\lambda}_{\kappa}(a U/ (V \rtimes H^T)) = |a|^{n-\dim V} \overline{\lambda}_{\kappa}(U/ (V \rtimes H^T))~,
\] with $n -\dim V >0$. This is only possible if $\overline{\lambda}_{\kappa}(U/ (V \rtimes H^T)) = \infty$. Now the result follows from \eqref{th.3}. \end{proof}
At first sight the condition $V \subsetneq \widehat{\mathbb{R}^n}$ used in the Corollary may appear somewhat artificial, or primarily motivated by the proof technique. It is in fact indispensable: Recall that in the case $V = \widehat{\mathbb{R}^n}$, with trivial choice for $H$, we obtain the Heisenberg group, which is a unimodular, strongly admissible translation complete group.  

It is worthwhile singling out the special case $U = \mathbb{R}^n$ of the corollary:
\begin{Corollary} \label{cor:unimod_nonadm}
Assume that $G=\mathbb{T}\times \mathbb{R}^n\times V \rtimes H^T$ is a unimodular translation complete subgroup of $G_{aWH}$, with $V \subsetneq \widehat{\mathbb{R}^n}$. Then $G$ is not admissible for $L^2(\mathbb{R}^n)$.
\end{Corollary}

We can also exclude discrete series subrepresentations in the nontrivial unimodular setting. The argument is another generalization of a result for the purely affine case, with a similar proof.  
{\Corollary \label{C.4.1} Assume that $G=\mathbb{T}\times \mathbb{R}^n\times V \times H$ is a unimodular translation complete subgroup of $G_{aWH}$, with $V \subsetneq \widehat{\mathbb{R}^n}$.  Then the quasi-regular representation $\pi$ does not contain a discrete series 
subrepresentation.}

\begin{proof}
Assuming the contrary, we obtain by Corollary \ref{cor:char_discrete_series} that the irreducible subspace is of the form $\mathcal{H}_{\mathcal{O}}$, for a single open $V \rtimes H^T$-orbit $\mathcal{O} \subset \widehat{\mathbb{R}^n}$. The desired contradiction will follow from Corollary \ref{cor:scaling_invariant}, once we have shown that $\mathbb{R}^+ \mathcal{O} = \mathcal{O}$.

For this purpose, denote by $U \subset \widehat{\mathbb{R}^n}$ the union of all open $V \rtimes H^T$-orbits. By Theorem definition, $U$ is an open subset. Furthermore, given any $\gamma \in \widehat{\mathbb{R}^n}$, $a>0$ and $(\xi, h) \in V \times H$ one has that 
\[  (a \cdot \gamma).(x,h) = h^{-T} (a \gamma - \xi) = a h^{-T} (\gamma - a^{-1} \xi)~.
\] Since $V$ is a vector space, this entails $(a \gamma).(V \rtimes H^T) = a (\gamma. (V \rtimes H^T))$. In particular, the orbit of $a \gamma$ is open precisely when the orbit of $\gamma$ is, for all $\gamma$, and this can be summarized by the statement that $\mathbb{R}^+ U = U$. 

Fixing any $\gamma \in \mathcal{O}$, $\mathbb{R}^+ \gamma$ is in fact  contained in a single \textit{connected component} of $U$. Since the open orbits partition $U$, and since $\gamma \in \mathcal{O}$, this already forces that $\mathbb{R}^+ \gamma \subset \mathcal{O}$, as desired. 
\end{proof}


The nonunimodular setting is in sharp contrast to the unimodular case, as it does not allow a canonical normalization of the measures in the measure decomposition. Hence there is also no natural finite measure condition that could be used to characterize  admissible group actions. As the following theorem shows, this is not due to a lack of appropriate proof techniques; instead, it turns out that the nonunimodular case is not bound by the same type of restrictions. 

\begin{theorem} \label{thm:adm_nonunimodular}
Let $G < G_{aWH}$ be a nonunimodular translation complete subgroup, and $\mathcal{H}_U \subset L^2(\mathbb{R}^n)$ a closed invariant subspace. Then $G$ is admissible for $\mathcal{H}_U$ if and only if $G$ is weakly admissible for $\mathcal{H}_U$. 
\end{theorem} 
\begin{proof} The "only-if" part is clear.
For the converse, suppose $\psi_0$ is a weakly admissible vector, and define $\Phi = \Phi_{\psi_0}$ as in Proposition \ref{L3.1}.
Pick a measure decomposition $(\overline{\lambda},\{\beta_{\mathcal{O}_\gamma}\}_{\gamma\in U})$ of Lebesgue measure $\lambda$ on $U$.
Since $G$ is nonunimodular, $\Delta_G$ is nontrivial on $V \rtimes H^T$, and there exists $(\xi_0,h_0)\in V \rtimes H^T$ 
such that $\Delta_G(\xi_0,h_0)<1/2$.\\
$\overline{\lambda}$ is a pseudo-image of $\lambda$, hence $\overline{\lambda}$ is a $\sigma$-finite measure on the orbit space, and we can write $U$ as a disjoint union $U=\bigcup_{n\in \mathbb{N}} V_n$ with 
$V_n\subset U$ Borel, $V\rtimes H^T$-invariant and with $\overline{\lambda}(V_n/(V\rtimes H))< \infty$.
Assume $\psi_0$ is a weakly admissible vector and 
\[\Phi(\gamma):= \left( \int_{V \rtimes H }|\widehat{\psi_0}(\gamma\cdot(\xi,h))|^2 d\mu_{V \rtimes H }(\xi,h)\right)^{1/2}.\]
Let $\widehat{\psi}(x)= \widehat{\psi_0}(x)/\Phi(x)$ and define
\[
 \Psi:\gamma\mapsto \left(\int_{\mathcal{O}_\gamma}|{\widehat{\psi}}(x)|^2 d\beta_{\mathcal{O}_\gamma}(x)\right)^{1/2}
\]
which is finite a.e. since $\widehat{\psi_0}\in L^2(\mathbb{R}^n)$ by Proposition \ref{L3.1}.

Additionally, we may assume $\Psi$ is bounded on each $V_n$. In particular, the functions $(1_{V_n}\ldotp \Psi)_{n\in\mathbb{N}}$ are square-integrable.
For all $V_n$, there exists $k_n\in \mathbb{N}$ such that 
\begin{equation}\label{eq:4-4}
  \frac{1}{2^{k_n}} ||1_{V_n}\cdot \Psi||_2^2< \frac{1}{2^n}.
\end{equation}
That is enough to pick $(\xi_0,h_0)$ such that $\Delta_{V\rtimes H}(\xi_0,h_0)<1/2$. 
Now define
\[
 \nu(\gamma):= \sum_{n\in \mathbb{N}}\Delta_{V\rtimes H}(\xi_0,h_0)^{k_n/2}\widehat{\psi}(\gamma\cdot(\xi_0,h_0)^{k_n}). 1_{V_n}(\gamma)
\]
for all $\gamma\in U $. In the following, we see that this function is the desired admissible vector:
\begin{equation*}
 \begin{aligned}
  \int_U |\nu(\xi)|^2 d\lambda(\xi)&=&&\int_{U/(V\rtimes H)}\int_{\mathcal{O}} |\nu(\xi)|^2 d\beta_{\mathcal{O}}(\xi)d\overline{\lambda}({\mathcal{O}})\\
   &=&&\sum_{n\in \mathbb{N}} \int_{V_n}\int_{\mathcal{O}}\Delta_{V\rtimes H}(\xi_0,h_0)^{k_n}|\widehat{\psi}(\gamma\cdot(\xi_0,h_0)^{k_n})|^2 d\beta_{\mathcal{O}}(\xi)d\overline{\lambda}({\mathcal{O}})\\
  &=&&\sum_{n\in \mathbb{N}} \int_{V_n}\int_{\mathcal{O}}\Delta_{V\rtimes H}(\xi_0,h_0)^{k_n}\delta(\xi_0,h_0)^{-k_n}|\widehat{\psi}(\xi)|^2 d\beta_{\mathcal{O}}(\xi)d\overline{\lambda}({\mathcal{O}})\\
  &=&&\sum_{n\in \mathbb{N}} \Delta_G(\xi_0,h_0)^{k_n}\int_{V_n}\int_{\mathcal{O}}|\widehat{\psi}(\xi)|^2 d\beta_{\mathcal{O}}(\xi)d\overline{\lambda}({\mathcal{O}})\\
  &=&&\sum_{n\in \mathbb{N}} \Delta_G(\xi_0,h_0)^{k_n} ||1_{V_n}\cdot \Psi||_2^2 \leq \sum_{n\in \mathbb{N}} 2^{-k_n}||1_{V_n}\cdot \Psi||_2^2<\infty,
 \end{aligned}
\end{equation*}
by choice of the $k_n$. Hence $\nu$ is square-integrable.

Moreover, the Calderon condition is also satisfied:
For a fixed $\gamma\in V_n$,
\begin{equation*}
 \begin{aligned}
 \int_{V\rtimes H} |\nu(\gamma\cdot(\xi,h))|^2d(\xi,h)&=&&\int_{V \rtimes H }|\psi(\gamma\cdot(\xi_0,h_0)(\xi,h))|^2 \Delta_{V\rtimes H}(\xi_0,h_0)d(\xi,h)\\
 &=&&\int_{V \rtimes H }|\psi(\gamma\cdot(\xi,h))|^2 d(\xi,h)=1. 
\end{aligned}
\end{equation*}
The inverse Plancherel transform of $\nu$ is then a well-defined element of $\mathcal{H}_U$, and it is admissible for $\mathcal{H}_U$.
\end{proof}

\begin{remark}
\label{rem:direct_sum}
As a consequence of Theorem \ref{thm:dec_open_orbits}, whenever $\pi$ has a discrete series subrepresentation, it already decomposes into finitely many irreducible subrepresentations corresponding to the finitely many open orbits. Not all of these are necessarily in the discrete series --for some open orbits the stabilizers may be noncompact--, but when they are, their sum $\pi$ is actually strongly square-integrable for the whole space $L^2(\mathbb{R}^n)$.

A direct way of seeing this is by realizing that for wavelets $\psi_1 \in \mathcal{H}_{\mathcal{O}_1}$ and $\psi_2 \in \mathcal{H}_{\mathcal{O}_2}$ associated to disjoint open orbits $\mathcal{O}_1, \mathcal{O}_2$, the associated wavelet transforms fulfill $\langle W_{\psi_1} f_1, W_{\psi_2} f_2 \rangle_{L^2(G)} = 0$, for all $f_1,f_2 \in L^2(\mathbb{R}^n)$, by a computation similar to the one proving Proposition \ref{L3.1}. One can then use this orthogonality relation to prove that picking admissible vectors for each irreducible subspace, with identical isometry constants, and taking their (finite) sum results in an admissible vector for the full representation space $L^2(\mathbb{R}^n)$. 

An indirect way of establishing strong admissibility uses Theorem \ref{thm:char_weak_adm} to conclude that it is weakly admissible, since there is a finite transversal of orbits with compact stabilizers. Furthermore, by Corollary \ref{C.4.1} the group is necessarily nonunimodular, and then the representation is indeed strongly square-integrable by Theorem \ref{thm:adm_nonunimodular}. 
\end{remark}

We now employ our characterizations to resume the discussion of various special cases that we already mentioned in the introductory sections. 
\begin{example} \label{ex:main_examples}
Combining $V = \widehat{\mathbb{R}^n}$ with any closed matrix group $H$ yields a translation complete subgroup $G = \mathbb{T} \times \mathbb{R}^n \times V \times H$. This group has a single open orbit $\mathcal{O} = \widehat{\mathbb{R}^n}$, and the stabilizer group associated to $0 \in \mathcal{O}$ is $\{ 0 \} \times H$. Hence it turns out that $G$ is weakly admissible if and only if it is admissible, if and only if $H$ is compact. The resulting group $G$ can then be understood as a compact extension of the reduced Heisenberg group.

In particular, the full affine Weyl-Heisenberg group fulfills neither admissibility condition, which is the main reason (already observed in \cite{Torresani1}) why one has to restrict the representation, using suitable subsets or cross-sections, in order to obtain inversion formulae and continuous frames. 

On the other hand, combining a proper subspace $V \subsetneq \widehat{\mathbb{R}^n}$ with a \textit{trivial} subgroup $H = \{ I_n \}$ yields a unimodular translation complete subgroup. The orbit space on $\widehat{\mathbb{R}^n}$ is equal to the quotient vector space $\widehat{\mathbb{R}^n}/V$, with trivial stabilizers, which establishes that $G$ is weakly admissible for $L^2(\mathbb{R}^d)$. Furthermore, one can show that the measure $\overline{\lambda}_{\kappa}$ normalized according to (\ref{eqn:meas_decomp_unimod}) in fact coincides with the standard Lebesgue measure on the quotient vector space $\widehat{\mathbb{R}^n}/V$ (which is positive dimension), hence is infinite. This entails that this group is not strongly admissible for $L^2(\mathbb{R}^n)$.

Finally, as we already noted above, the case where $V = \{ 0 \}$ corresponds precisely to the fairly well studied setting of generalized, affine wavelets \cite{Taylor,Fuehr4,MR3319960,MR3319960}. Hence the truly novel constructions arise from a combination of $\{ 0 \} \subsetneq V \subsetneq \widehat{\mathbb{R}^n}$ with $\{ I_n \} \subsetneq H \subsetneq GL(n,\mathbb{R})$. Section \ref{sect:examples} below will present a diverse list of examples in dimensions two and three.
\end{example}

\begin{remark}
    The sharp dichotomy between the unimodular and nonunimodular cases, expressed in Theorems \ref{th.3} and \ref{thm:adm_nonunimodular}, is analogous to the results in \cite{MR1908919}. This source studies subrepresentations of the regular representation $\lambda_G$ of a locally compact group, assuming $\lambda_G$ is of type I, and shows that such subrepresentations allow an admissible vector if and only if the group is unimodular, and the associated subspace fulfilled a certain finite measure condition (\cite[Theorem 1.6] {MR1908919}, comparable to our Theorem \ref{th.3}), or if the group is nonunimodular, with no additional conditions imposed (\cite[Corollary 0.3]{MR1908919}, comparable to our Theorem \ref{thm:adm_nonunimodular}). 

    This analogy is not a coincidence, and has been successfully studied and deepened in the affine case, by explicitly computing the Plancherel measure of $G$ and relating it to the measure decomposition used in this section; see \cite{Fuehr4} for details. We expect that this can also be achieved for the more general setting studied here, but that endeavor goes well beyond the goals of the current paper. 
\end{remark}

\section{Reproducing subgroups and the metaplectic representation}

\label{sect:metaplectic}
The metaplectic group and its subgroups have been recognized as another convenient setting for the discussion of generalized wavelet systems and associated reproducing formulae, encompassing both the windowed Fourier transform and the purely affine setting as well as a variety of intermediate cases, see e.g. \cite{Cordero,DeMari,Tabacco,Alberti}. In this section we will show how our setting fits into the landscape of the metaplectic group, and compare our admissibility criterion to the Wigner-transform based criterion derived in \cite{DeMari}.

\subsection{Metaplectic representation}

The metaplectic representation is a projective representation of the \textit{symplectic group} that arises from the action of this group on the Heisenberg group. The following exposition of the symplectic group, the metaplectic representation and the \textit{extended} metaplectic representation are somewhat sketchy; we refer to the sources \cite{Folland2,DeMari} for more background information.  

By definition, the \textit{symplectic group} is given as 
\[
 Sp(n,\mathbb{R})=\{A\in GL(2n,\mathbb{R}): A^t\mathcal{J}A=\mathcal{J}\},
\]
where
\[
 \mathcal{J}=\left({\begin{array}{cc}
0 & I_n\\
I_n & 0
\end{array}}
\right).
\]

Direct calculation shows that $B \in Sp(n,\mathbb{R})$ acts on the (reduced) Heisenberg group $\mathbb{H} = \mathbb{T} \times \mathbb{R} \times \widehat{\mathbb{R}^n} \subset G_{aWH}$ by group automorphisms defined by 
\begin{equation} \label{eqn:sympl_autom}
(z,x,\xi) \mapsto (z,B(x,\xi))~;
\end{equation} Here we identify $(x,\xi) \in \mathbb{R}^n \times \widehat{\mathbb{R}^n}$ with a vector in $\mathbb{R}^{2n}$, by simple concatenation. If we let $Z = \mathbb{T} \times \{ 0 \} \times \{0 \} \subset\mathbb{H}^n$ denote the center of $\mathbb{H}^n$, then the above map turns is an automorphism that leaves $Z$ pointwise fixed. As a consequence of this fact, together with the Stone-von-Neumann theorem and Schur's Lemma, one obtains a map
\[
\mu: Sp(n,\mathbb{R}) \to \mathcal{U}(L^2(\mathbb{R}^n))
\] such that 
\[
\forall (z,x,\xi) \in \mathbb{H}^n \forall A \in  Sp(n,\mathbb{R})~:~ \pi(z,A(x,\xi)) = \mu(A) \pi(z,x,\xi) \mu(A)^*~. 
\]
Note that strictly speaking this equation only determines $\mu(A)$ up to complex scalars of modulus one; there are however special choices available \cite{Folland2}. The map $\mu: A \mapsto \mu(A)$ is the \textit{metaplectic representation} of $Sp(n,\mathbb{R})$. It is only a \textit{projective} representation, i.e. one has $\mu(A) \mu(B) \mu(AB)^* \in \mathbb{T} \cdot \mathrm{id}_{L^2(\mathbb{R}^n)}$, for all $A,B \in Sp(n,\mathbb{R})$.

We next consider the \textit{extended metaplectic group}, i.e. the semidirect product $\mathcal{G} = \mathbb{H}^n \rtimes SP(n,\mathbb{R})$, which is defined using the action (\ref{eqn:sympl_autom}). Then the definition of the metaplectic representation immediately entails that the \textit{extended metaplectic representation} $\mu_e: \mathbb{H}^n \rtimes Sp(n,\mathbb{R}) \to \mathcal{U}(L^2(\mathbb{R}^n))$, defined by 
\[
\mu_e(z,x,\xi, B) = \pi(z,x,\xi) \mu_e(B)~
\] is again a (strongly continuous) projective representation. 

The affine Weyl-Heisenberg group and its canonical representation fits into this framework as follows: The map $\beta: {\rm GL}(n,\mathbb{R})\hookrightarrow Sp(n,\mathbb{R})$, defined by
$h\mapsto \left({\begin{array}{cc}
h & 0\\
0 & h^{-T}
\end{array}}
\right)$ is easily verified to be an injective group homomorphism. This then gives rise to another group embedding $\alpha: G_{aWH} \hookrightarrow  \mathbb{H}^n \rtimes Sp(n,\mathbb{R})$, via $\alpha(z,x,\xi,h) = ((z,x,\xi),\beta(h))$. As a further pleasant surprise, one can verify (possibly after adjusting phase terms in the definition of $\mu_e$) the equation
\[
\mu_e (\beta(h)) = D_h = \pi(1,0,0,h)~, 
\] see e.g. \cite[Section 3.3]{Tabacco}. This finally leads to the realization that 
\begin{equation} \label{eqn:meta_vs_canon}
\mu_e \circ \alpha = \pi~. 
\end{equation}In summary, \textit{there is a canonical embedding $G_{aWH} \hookrightarrow \mathbb{H}^n \rtimes Sp(n,\mathbb{R})$, and the restriction of $\mu_e$ to the subgroup $G_{aWH}$ coincides with $\pi$!} 

Here we need to point out that the cited sources \cite{DeMari,Tabacco,Cordero,Alberti} work with slightly different versions of the Heisenberg group; they either use the \textit{simply connected} Heisenberg group (which has noncompact center), or the quotient group $\mathbb{R}^{2n} \cong \mathbb{H}^n/Z$, whereas we prefer to use the Heisenberg group with compact center. These differences, however, are immaterial for the following considerations. 

\subsection{Reproducing subgroups and associated admissibility conditions}

Based on the preliminaries given in the previous section, we can now present the pertinent definition of reproducing subgroups from \cite{DeMari}.
\begin{definition} \label{defn:repro_metaplectic}
    A connected Lie subgroup $H$ of $\mathbb{H}^n \rtimes Sp(n,\mathbb{R})$ is called  \emph{reproducing subgroup} for $\mu_e$ if there exists a function 
$\psi\in \rm L^2(\mathbb{R}^n)$ such that 
\begin{equation}\label{E5.1}
 \| f \|_2^2 = C_\psi \int_H \left| \langle f,\mu_e(h)\psi\rangle \right|^2 dh, ~~~\forall f\in \rm L^2(\mathbb{R}^n),
\end{equation}
with a constant $C_\psi>0$.
A function $\psi$ giving rise to this norm equality is called \emph{reproducing function}.
\end{definition}

As a consequence of the canonical embedding $G_{aWH} \hookrightarrow \mathbb{H}^n \rtimes SP(n,\mathbb{R})$ and the compatibility of extended and canonical representations (see above), we immediately see that a (connected) translation complete subgroup $G = \mathbb{T} \times \mathbb{R}^n \times V \times H < G_{aWH}$ is admissible for $L^2(\mathbb{R}^n)$ if and only if it is reproducing in the sense of Definition \ref{defn:repro_metaplectic}, when viewed as a subgroup of $\mathbb{R}^n \rtimes SP(n,\mathbb{R})$. Furthermore, the reproducing functions are precisely the admissible vectors. 

It is reasonable to expect that the connectedness assumption in Definition \ref{defn:repro_metaplectic} can be discarded without jeopardizing any of the subsequent arguments; however, since this assumption was always made in the previous sources, we will adopt it for the remainder of this section. 
We next describe an admissibility criterion  for subgroups of $\mathbb{H}^n \rtimes Sp(n,\mathbb{R})$ via the Wigner
distribution, as provided in \cite{DeMari}. For this purpose, we first need the definition of the Wigner distribution itself. Given $f \in L^2(\mathbb{R}^n)$, its Wigner distribution is the function $W_f: \mathbb{R}^n \times \mathbb{R}^n \to \mathbb{C}$, defined as 
\[
W_f(x,\xi) = \int_{\mathbb{R}^n} e^{-2\pi i \langle \xi,y \rangle}f(x+\frac{y}{2})\overline{f(x-\frac{y}{2})} dy~. 
\]

We also need the canonical action of $\mathbb{H}^n \rtimes Sp(n,\mathbb{R})$ on $\mathbb{R}^{n} \times \mathbb{R}^n$, which is given by 
\[
(z,x,\xi,h) . (y_1,y_2) = h(y_1,y_2) + (x,\xi)~, (z,x,\xi,h) \in \mathbb{T} \times \mathbb{R^n} \times \widehat{\mathbb{R}^n} \times Sp(n,\mathbb{R})
\] where we again identified vectors in $\mathbb{R}^n \times \widehat{R}^n$ with elements of $\mathbb{R}^{2n}$.

One then obtains the following Wigner-function based admissibility condition \cite[Theorem 1]{MR2224394}:
\begin{theorem} \label{P:4-1}  Assume that $H < \mathbb{H}^n \rtimes Sp(n,\mathbb{R})$ is a connected subgroup. Let $\psi\in \rm L^2(\mathbb{R}^n)$ be such that the mapping
\begin{equation}\label{eq:4-1}
  h\mapsto W_{\mu_e(h)\psi}(x,\xi)=W_\psi(h^{-1}\cdot(x,\xi))
\end{equation}
is in $\rm L^1(H)$ for a.e. $(x,\xi)\in \mathbb{R}^{2n}$, and 
\begin{equation}\label{eq:4-2}
 \int_H  |W_\psi(h^{-1}\cdot(x,\xi))|dh\leq M, ~~~~\text{ for a.e. } (x,\xi)\in \mathbb{R}^{2n}. 
\end{equation}
Then, the reconstruction formula \eqref{E5.1} holds for all $f\in \rm L^2(\mathbb{R}^n)$ if and only if 
\begin{equation}\label{eq:4-3}
  \int_H  |W_\psi(h^{-1}\cdot(x,\xi))|dh=1, ~~~~ \text{ for a.e. } (x,\xi)\in \mathbb{R}^{2n}.
\end{equation}
\end{theorem}

\begin{remark} We note that the technical assumptions \eqref{eq:4-1} and \eqref{eq:4-2} are not necessary for the characterization of reproducing vectors; there exist reproducing vectors that do not fulfill these requirements. More details are given in \cite{Tabacco}.
 \end{remark}

\subsection{Comparing admissibility conditions}

We can now work out how the different admissibility criteria for (connected) translation complete subgroups $G = \mathbb{T} \times \mathbb{R}^n \times V \times H$ relate to each other. The following result provides a direct link between the functions on which these criteria are based. The result was already known for the affine case; see \cite[Proposition 9]{Tabacco}.

\begin{Proposition} \label{ref:prop_adm_cond_coincide} Let $G = \mathbb{T} \times \mathbb{R}^n \times V \times H$ be a connected translation complete subgroup of $G_{aWH}$, embedded canonically into $\mathbb{H}^n \rtimes SP(n,\mathbb{R})$. Let $\psi\in L^1(\mathbb{R}^n)\cap L^2(\mathbb{R}^n)$. Then the wavelet admissibility condition and the Wigner admissibility condition
coincide, that is
\begin{equation} \label{eqn:adm_cond_eq}
 \int_{G} W_\psi((z,p,q,h)^{-1}\cdot(x,\xi))d(z,p,q,h){=}\int_{V\rtimes H}  |\widehat\psi(h^T(\xi-q))|^2 d\mu(q,h)
\end{equation}
holds for all $x \in \mathbb{R}^n$ and almost all $\xi\in \mathbb{R}^{n}$. 
\end{Proposition}
\begin{proof}
    If $\psi\in L^1(\mathbb{R}^n)\cap L^2(\mathbb{R}^n)$, we have $W_\psi\in \mathcal{C}(\mathbb{R}^{2n})$ and 
$W_\psi(.,\xi)\in \rm L^1(\mathbb{R}^{n})$, for each $\xi\in \mathbb{R}^n$, and the so-called Wigner marginal property 
\begin{equation*}
 \int_{\mathbb{R}^n} W_f(x,\xi)d\xi=|f(x)|^2,~~~~\int_{\mathbb{R}^n} W_f(x,\xi)dx=|\widehat f(\xi)|^2
\end{equation*} 
holds, see \cite[1.96]{Folland2}. The following straightforward calculation shows the coincidence:
\begin{eqnarray*}
  \lefteqn{\int_G W_\psi((z,p,q,h)^{-1}\cdot(x,\xi))d(p,q,h)} \\ &=& \int_V\int_H\int_{\mathbb{R}^n} W_\psi(h^{-1}(x-p),h^T(\xi- q))dpdq|{\rm det}(h^T|_V)|\frac{dh}{|{\rm det}(h)|}\\
 &=& \int_V\int_H\int_{\mathbb{R}^n} W_\psi(u,h^T(\xi- q)) du dq |{\rm det}(h^T|_V)| dh\\
 &=& \int_V\int_H\int_{\mathbb{R}^n}\int_{\mathbb{R}^n} e^{-2\pi i \langle h^T(\xi-q),y \rangle}\psi(u+\frac{y}{2})\overline{\psi(u-\frac{y}{2})} dy du dq |{\rm det}(h^T|_V)| dh\\
 &=& \int_V\int_H\int_{\mathbb{R}^n} e^{-2\pi i \langle h^T(\xi-q),y \rangle}\left(\int_{\mathbb{R}^n} \psi(u)\overline{\psi(u-y)}du\right)  dy dq |{\rm det}(h^T|_V)| dh\\
 &=&\int_V\int_H\int_{\mathbb{R}^n} e^{-2\pi i \langle h^T(\xi-q),y \rangle}\left( \widehat{\psi\ast \psi^\ast}(y)\right) dy dq |{\rm det}(h^T|_V)| dh\\
 &=& \int_V\int_H (\widehat{\psi\ast \psi^\ast})(h^T(\xi-q)) dq |{\rm det}(h^T|_V)| dh\\
 &=& \int_{V\rtimes H}|\widehat\psi(h^T(\xi-q))|^2 d\mu(q,h).
\end{eqnarray*}
\end{proof}
We close this comparison with several observations: Firstly, our admissibility condition in Corollary \ref{C3.1} is sharper in the sense that it does not need additional technical assumptions (such as (\ref{eq:4-1}) and (\ref{eq:4-2}) above), which are not strictly related to admissibility. Secondly, the question whether vectors exist which fulfill these criteria appears to be very difficult in the case of Theorem \ref{P:4-1}; in general, constructing functions $f$ with the aim of guaranteeing certain properties of their Wigner functions is a difficult undertaking. By contrast, the results from Section \ref{sect:adm_criteria} allow to assess, in a sharp way, which groups allow the existence of vectors fulfilling our admissibility criteria. 

 The obvious counterpoint to these observations is that, by construction, our preferred setting of translation complete subgroups is more restrictive than \cite{Cordero,DeMari,Tabacco}, and more stringent results are therefore to be expected. 

\section{Concrete examples in low dimensions}\label{sect:examples}

In this section, we present a list of novel examples in dimensions two and three, with the aim of showing that there is a variety of yet unexplored wavelet systems that arise from these group actions. The examples also demonstrate that the criteria developed in Section \ref{sect:adm_criteria} allow an efficient but nuanced appraisal of the different cases. The fact that most of the examples given below actually cover whole \textit{families} of groups, parameterized by one or several real parameters, emphasizes that this novel class of constructions is quite rich. 

We also want to point out that, to the extent that we have been able to verify, none of the following examples are covered in the previous sources \cite{DeMari,Cordero,MR2224394,Tabacco,Alberti,MR3008561}.

As explained in Example \ref{ex:main_examples}, we focus on translation complete subgroups $G = \mathbb{T} \times \mathbb{R}^n \times V \times H$ with subspaces $\{ 0 \} \subsetneq V \subsetneq \widehat{\mathbb{R}^n}$ and $\{ I_n \} \subsetneq H \subsetneq GL(n,\mathbb{R})$. In the interest of restricting this discussion to special choices of $V$, it is beneficial to first clarify the role of conjugacy within $G_{aWH}$.
Recall that two subgroups $G_1,G_2$ of the same group $G$ are called \textbf{conjugate within} $G_{a}$ if there exists $x \in G$ with $x^{-1} G_1 x = G_2$. 
The following lemma records how the properties that we are interested in behave under conjugacy. 
\begin{lemma}\label{lem:covariance_conjugacy}
    Let $G_1 = \mathbb{T} \times \mathbb{R}^n \times V_1 \times H_1$ and $G_2 = \mathbb{T} \times \mathbb{R}^n \times V_2 \times H_2$ be translation complete subgroups. 
    \begin{enumerate}
        \item[(a)] $G_1$ and $G_2$ are conjugate within $G_{aWH}$ if and only if there exists $g \in GL(n,\mathbb{R})$ such that $(1,0,0,g)^{-1} G_1 (1,0,0,g) = G_2$. This equation is equivalent to the equalities $V_2 = g^{T} V_1$ and $H_2 = g^{-1} H_1 g$ holding simultaneously.
        \item[(b)] Assume that $(1,0,0,g)^{-1} G_1 (1,0,0,g) = G_2$. Let $U \subset \widehat{\mathbb{R}^n}$ be $V_1 \rtimes H_1^T$-invariant. Then $g^{T} U \subset \widehat{ \mathbb{R}^n}$ is $V_2 \rtimes H_2^T$-invariant. If $\psi \in \mathcal{H}_U$ is (weakly) admissible for the action of $G_1$ on $\mathcal{H}_U$, then $\pi(1,0,0,g) \psi$ has the same property for the action of $G_2$ on $\mathcal{H}_{g^T U}$. The quasi-regular representation of $G_1$, acting on $\mathcal{H}_U$, is a discrete series representation if and only if the quasi-regular representation of $G_2$ acting on $\mathcal{H}_{g^T U}$ has the same property        \item[(c)] Let $G = \mathbb{T} \times \mathbb{R}^n \times V \times H$ be a translation complete  subgroup. Then $G$ is conjugate within $G_{aWH}$ to a translation complete subgroup $G' = \mathbb{T} \times \mathbb{R}^n \times \left( \mathbb{R}^k \times \{ 0 \}^{n-k} \right) \times H'$, where $k = \dim(V)$, and $H'<GL(n,\mathbb{R})$ is suitably chosen. 
    \end{enumerate}
\end{lemma}
\begin{proof}
    We recall that the Heisenberg group $\mathbb{H}^n$ is normal. In fact, direct computation establishes that elements of $\mathbb{H}^n$ normalize any translation complete subgroup. But this means that 
    \[
    (z,x,\xi,g)^{-1} G_1 (z,x,\xi,g) = (1,0,0,g)^{-1} G_1 (1,0,0,g)~,
    \] which establishes the first statement of (a). The second one follows from 
    \begin{eqnarray*}
    (1,0,0,g)^{-1} (z,x,\xi,h) (1,0,0,g)  & = & (1,0,0,g^{-1}) (z,x,\xi,h) (1,0,0,g)  \\ & = &  (z,g^{-1} x,g^T \xi, g^{-1}hg)~.
    \end{eqnarray*}
    For part (b), pick $\omega = g^T \omega_0 \in g^T U$, as well as $(\xi',h') = (g^T\xi, g^{-1} h g) \in V_2 \times H_2$, i.e., with $(\xi,h) \in V_1 \times H_1$.
    Then the dual action of $(\xi',h')$ on $\omega$ simplifies to 
    \[
    \omega. (\xi',h') = (g^{-1} h g)^T (\omega - g^T \xi) = g^T \underbrace{h^T (\omega_0 - \xi)}_{\in U}
    \] which shows that conjugation by $(1,0,0,g)$ intertwines the dual actions of $V_1 \rtimes H_2^T$ on $U$ and of $V_2 \rtimes H_2^T$ on $g^T U$. This shows the first statement of part (b). For the second statement, we use the quasi-regular representation of the full affine Weyl-Heisenberg group to define the unitary operator $T = \pi(1,0,0,g)$. Then by direct calculation, similar to the one we just performed, $T(\mathcal{H}_{g^T U}) = \mathcal{H}_U$, and given any $(z_1,x_1,\xi_1,h_1) \in G_1$ and every $(z_2,x_2,\xi_2,h_2) = (1,0,0,g^{-1}) (z_1,x_1,\xi_1,h_1) (1,0,0,g)$, we find
    \[
    \pi(z_2,x_2,\xi_2,h_2) = T^{-1} \pi(z_1,x_1,\xi_1,h_1) T~.
    \]    At the level of matrix coefficients this entails for $f, \psi \in \mathcal{H}_{g^T
 U} $
    \[
    W_{\psi} f  (z_2,x_2,\xi_2,h_2) = \langle f, T^{-1} \pi(z_1,x_1,\xi_1,h_1) T \psi \rangle = W_{T \psi} (Tf)(z_1,x_1,\xi_1,h_1)~. 
    \] In particular, since conjugation induces a topological group isomorphism between $G_1$ and $G_2$, uniqueness of Haar measure entails 
    \[
    \| W_\psi f \|_{L^2(G_2)} = C_g \| W_{T \psi} Tf \|_{L^2(G_1)}~.
    \] The remaining statements from part (b) directly follow from these observations. 
    For part (c) it remains to pick $g \in GL(n,\mathbb{R})$ such that $ g^T \left( \mathbb{R}^k \times \{ 0 \}^{n-k}\right) = V$, and then let $H' = g^{-1} H g$.
\end{proof}


\subsection*{Admissible subgroups in dimension 2}
Following the observations in Example \ref{ex:main_examples} and Lemma \ref{lem:covariance_conjugacy}, we will focus on the case $V = \mathbb{R} \times \{ 0 \} \subset \mathbb{R}^2$.  Our primary goal is to identify matrix groups $H$ that are compatible with $V$, i.e., fulfill $H^T V = V$, and such that the dual action has a single open orbit, with associated compact stabilizers. 
\begin{example} \label{ex:dim_two}
    Let $V=\mathbb{R}\times \{0\}$ and $H=\{\exp(t\mathbb{X}):t\in \mathbb{R}\}$, where $\mathbb{X}=\left({\begin{array}{cc}
a & 0\\
c & d    
\end{array}}
\right),a,c,d\in \mathbb{R}$ are fixed with $d\neq 0$. Without loss of generality, let $d=1$.
We let \[H =\left\{\left({\begin{array}{cc}

e^{ta} & 0\\
f(t) & e^t    

\end{array}}
\right) :t\in \mathbb{R}\right\},\]
where 
\[
f(t)=tc+\frac{t^2c}{2!}[1+a]+\frac{t^3c}{3!}[1+a+a^2]+\frac{t^4c}{4!}[1+a+a^2+a^3]+
...~.
\]
As explained in example \ref{ex:block_structure}, $V$ is invariant under $H^T$, i.e., $H^TV\subset V$.\\
Now consider the orbit $\mathcal{O}_{(1,1)}$.
To see the action of $V\rtimes H^T$ on $\mathbb{R}^2$,
let $\gamma=\left({\begin{array}{c}

\gamma_1\\
\gamma_2
\end{array}}
\right)\in U$, $\xi=\left({\begin{array}{c}

\xi_1\\
0
\end{array}}
\right)\in V$ and $h=\left({\begin{array}{cc}

e^{ta} & 0\\
f(t) & e^t    
\end{array}}
\right)\in H_0$. 
So, 
\[
 h^T(\gamma-\xi)= \left({\begin{array}{cc}

e^{ta} & f(t)\\
0 & e^t    

\end{array}}
\right)
\left({\begin{array}{c}
\gamma_1 - \xi_1\\
\gamma_2
\end{array}}
\right)= \left({\begin{array}{c}
e^{ta}(\gamma_1-\xi_1)+ f(t)\gamma_2\\
e^t\gamma_2
\end{array}}
\right).
\]
Thus, 
\[
\mathcal{O}_{1,1}=\left\{ \left({\begin{array}{c}

e^{ta}-e^{ta}\xi_1+ f(t)\\
e^t
\end{array}}\right): \xi_1, t\in \mathbb{R}
\right\}.\]
Clearly, a suitable (in fact: unique) choice of $t$ solves the equation 
$y = e^t$, for given $y \in \mathbb{R}^+$, and once $t$ is fixed accordingly, $\xi_1 \in \mathbb{R}$ can be determined uniquely to solve 
\[ x = e^{ta}-e^{ta}\xi_1+ f(t) \,\,
\] for any given $x \in \mathbb{R}$. This shows that $V \rtimes H^T$ acts freely on the  open orbit $\mathcal{O}_{1,1} = \mathbb{R} \times \mathbb{R}^+$, thereby proving (via Corollary \ref{cor:char_discrete_series}) that the subrepresentation acting on $\mathcal{H}_{\mathcal{O}_{(1,1)}}$ is in the discrete series. 

By the same argument we get $\mathcal{O}_{(1,-1)} = \mathbb{R} \times (- \mathbb{R}^+)$, and therefore the subspace $\mathcal{H}_{\mathcal{O}_{(1,-1)}}$ carries a discrete series subrepresentation as well. In summary, $L^2(\mathbb{R}^2) = \mathcal{H}_{\mathcal{O}_{(1,1)}} \oplus \mathcal{H}_{\mathcal{O}_{(1,1-)}}$  is the sum of two discrete series subrepresentations. 

We finally determine the admissibility condition. Suppose $\psi\in \mathcal{H}_{\mathcal{O}_{1,1}}$ and $\eta=h^T(\gamma-\psi)$. 
By Theorem \eqref{T3.1}, $\psi$ is admissible if and only if 
\[
 \int_{\mathcal{O}_{1,1}} |\widehat{\psi}(\eta)|^2 \Psi(\eta) d\lambda(\eta)<\infty,
\]
where $$\Psi(h^T(\gamma-\xi)):=\frac{|{\rm det} (h^T|_V)|}{\Delta_H(h)|{\rm det}(h)|}.$$
In order to make this even more explicit, we first observe that the one-parameter group $H$ is commutative, 
and ${\rm det}(h^T|_V)=e^{ta}$, 
hence
$\Psi(\eta)=\frac{1}{e^t}$. 
Therefore,
\[
 \psi \mbox{ is admissible } \Leftrightarrow \int_{\mathbb{R}\times \mathbb{R}^*} |\widehat{\psi}(e^{ta}-e^{ta}\xi_1+ f(t),e^t)|^2\frac{1}{e^t} dtd\xi_1<\infty.
\]
This can be made even more explicit by employing
\begin{equation*}
 \begin{aligned}
   f(t)&=&& c \sum_{k=1}^{\infty} \frac{1}{k!}t^k (\sum_{m=0}^{k-1}a^m)\\
   &=&& c \sum_{k=1}^{\infty} \frac{1}{k!}t^k (\frac{a^k-1}{a-1}), \hspace{0.7cm} a\neq 1\\
   &=&& \frac{c}{a-1}(e^{ta}-e^t).
 \end{aligned}
\end{equation*}
\end{example}
 It can be shown that the dilation groups $H$ given above cover all suitable nontrivial examples of dilation groups that give rise to translation complete subgroups $G = \mathbb{T} \times \mathbb{R}^n \times V \times H$, with $V=\mathbb{R}\times \{0\}$. All other closed connected subgroups $H \subset GL(2,\mathbb{R})$ either fail the compatibility condition $H^T V \subset V$, or do not have open orbits with compact stabilizers. We refer to the thesis \cite{rashidi2020reproducing} of the first author for additional explanations. 

\subsection*{ Admissible subgroups in dimension 3.}
Applying again the reasoning of Example \ref{ex:main_examples} and Lemma \ref{lem:covariance_conjugacy}, we focus on the cases $V_j = \mathbb{R}^j \times \{ 0 \}^{3-j} \subset \widehat{\mathbb{R}^3}$, for $j=1,2$.
The criterion established in Example \ref{ex:block_structure} shows that $h^T V_i \subset V_i$ is equivalent to the block structure
\[
h = \left( \begin{array}{cc} h_{11} & 0 \\ h_{21} & h_{22} \end{array} \right)\,\,,
\] where $h_{11}$ is a suitable $1 \times 1$ block for $V_1$, and $0$ is of size $1 \times 2$, while $h_{11}$ is of size $2 \times 2$ for $V_2$, and $0$ of size $2 \times 1$. We will now give several families of dilation groups, check their compatibility with $V_1$ and $V_2$, and comment on admissibility (in various degrees). Similarly to Example \ref{ex:dim_two}, the explicit determination of admissibility criteria is a more or less straightforward computational task for each of the examples presented in the following; however, we will focus on determining (weak) admissibility. 

\begin{example} \label{E1}
 Assume $0\neq \beta$, $\alpha\in\mathbb{R}$ are fixed. Both $V_1$ and $V_2$ are invariant under
\[H =\left\{\left({\begin{array}{ccc}
e^{\alpha x} & 0 & 0\\
0 & e^{\beta x}    & 0\\
0 & 0 & e^{ x}
\end{array}}
\right): x\in \mathbb{R}\right\},
\]
i.e., $H^TV_1\subset V_1$ and $H^TV_2\subset V_2$.
Furthermore, $V_1 \rtimes H^T$ acts on $\mathbb{R}^3$ by
\[
(\gamma,(\xi,h))\mapsto
\left({\begin{array}{ccc}
e^{\alpha x}(\gamma_1-\xi_1) \\
e^{\beta x} \gamma_2 \\
e^{ x}\gamma_3 
\end{array}}
\right),
\]
This action can be seen to be free on the open conull subset $U = \mathbb{R} \times \mathbb{R}^* \times \mathbb{R} \subset \mathbb{R}^3$. However, already for dimension reasons, there cannot be an open orbit, since $\dim(V_1) + \dim(H) = 2 < 3$ (recall Lemma \ref{L1}). Hence, the quasi-regular representation has no discrete series subrepresentations. However, it is easy to verify that $\{0 \} \times \{ \pm 1 \} \times \mathbb{R} \subset U$ is a Borel transversal for orbit space $U/V_1 \rtimes H^T$. Hence, the group $G = \mathbb{T} \times \mathbb{R}^n \times V_1 \times H$ is weakly admissible by Theorem \ref{thm:char_weak_adm}. By Corollary \ref{cor:unimod_nonadm} and Theorem \ref{thm:adm_nonunimodular}, it is strongly admissible if and only if it is nonunimodular, and this is the case if and only if $\alpha + \beta \not= -1$.

Regarding $V_2$, we first note that $H$ is compatible with $V_2$. Furthermore, the action of $V_2 \rtimes H_0^T$ has open orbits. In fact, picking any $\xi = (1,\epsilon_1, \epsilon_2 )$, with $\epsilon_1,\epsilon_2 \in \{ \pm 1 \}$, a computation similar to the one carried out in Example \ref{ex:dim_two} establishes that $V_2 \rtimes H^T \xi = \mathbb{R} \times \epsilon_1 \mathbb{R}^+ \times \epsilon_2 \mathbb{R}^+$, and that the action on these open orbits is again free. Hence, the quasi-regular representation decomposes into four subspaces on which the representation acts via discrete series representations. 

\end{example}

\begin{example} \label{E3} Let $0\neq\beta$, $\alpha$ be fixed, and
\[H=\left\{\left({\begin{array}{ccc}
e^{\alpha x} & xe^{\alpha x} & 0\\
0 & e^{\alpha x}    & 0\\
0 & 0 & e^{\beta x}
\end{array}}
\right): x\in \mathbb{R} \right\}.
\]
$V_1$ is not invariant under $H^T$, but $H^TV_2\subset V_2$. $V_2 \rtimes H^T$ acts on $\mathbb{R}^3$ by
\[
(\gamma,(\xi,h))\mapsto
\left({\begin{array}{ccc}
e^{\alpha x}(\gamma_1-\xi_1)+xe^{\alpha x} (\gamma_2 -\xi_2)\\
e^{\alpha x} (\gamma_2-\xi_2) \\
e^{\beta x}\gamma_3 
\end{array}}
\right)\,\,.
\]
This action has two open orbits $\mathbb{R}^2 \times \pm \mathbb{R}^+$, where it is free. Hence, the representation of the associated translation complete subgroup decomposes into two discrete series representations. 
\end{example}

\begin{example} \label{E4} Let $\alpha\in \mathbb{R}$  be fixed, and
\[H=\left\{\left({\begin{array}{ccc}
e^{\alpha x} & 0 & 0\\
0 & e^{ x}    & 0\\
0 & 0 & e^{y}
\end{array}}
\right): x,y\in \mathbb{R} \right\}.
\]
 $V_1$, $V_2$ are invariant under $H^T$, and $V_1 \rtimes H^T$  acts on $\mathbb{R}^3$ by
\[
(\gamma,(\xi,h))\mapsto
\left({\begin{array}{ccc}
e^{\alpha x}(\gamma_1-\xi_1) \\
e^{ x} \gamma_2 \\
e^{ y}\gamma_3 
\end{array}}
\right).
\]
We obtain four open orbits
$ 
\mathcal{O}_{(1,\pm 1, \pm 1)}
 = \mathbb{R} \times   \pm \mathbb{R}^+ \times  \pm \mathbb{R}^+.
$
The action on each orbit is free, again yielding a decomposition of the canonical representation into four discrete series representations. 
By contrast, the action of $V_2 \rtimes H^T$ has two open orbits, but noncompact stabilizers everywhere on $\mathbb{R}^3$. This shows that the associated representation decomposes into irreducible subrepresentations that are not in the discrete series; no nonzero matrix coefficient is square-integrable in this setting. 
\end{example}
\begin{example}\label{E2} Let
\[H=\left\{\left({\begin{array}{cc|c}
e^{\alpha x}\cos(x) & e^{\alpha x}\sin(x)  & 0\\
-e^{\alpha x}\sin(x) & e^{\alpha x}\cos(x)    & 0\\ \hline
0 & 0 & e^{ x}
\end{array}}
\right): x\in \mathbb{R}, \alpha=0,1 \right\}.
\]
$V_1$ is not invariant under $H^T$ but $V_2$ is. Also,
$V_2 \rtimes H$ acts on $\mathbb{R}^3$ by
\begin{equation*}
\begin{aligned}
(\gamma,(\xi,h))\mapsto
\left({\begin{array}{ccc}
e^{\alpha x}\cos(x)(\gamma_1-\xi_1)+ e^{\alpha x}\sin(x)(\gamma_2 -\xi_2) \\
-e^{\alpha x}\sin(x)(\gamma_1-\xi_1)+ e^{\alpha x}\cos(x) (\gamma_2-\xi_2) \\
e^{ x}\gamma_3 
\end{array}}
\right),
\end{aligned}
\end{equation*}
which shows the existence of two open orbits 
$
 \mathcal{O}_{(1,1,\pm 1)}= \mathbb{R} \times  \mathbb{R} \times  \pm \mathbb{R}^+\,\,.
$
The action on these orbits is free, hence we obtain a decomposition of the representation into two discrete series subrepresentations. 
\end{example}
{\remark All other elements of the catalog of subgroups of $GL(3,\mathbb{R})$ given in \cite{Wick} fail the compatibility or the (weak) admissibility condition, for both $V_1$ and $V_2$.}

\begin{remark}
    We remind the reader that in all cases where $\mathbb{R}^n$ decomposes into finitely many open free orbits, the representation acting on all of $L^2(\mathbb{R}^n)$ is in fact strongly square integrable, see Remark \ref{rem:direct_sum}.
    In fact, for each of the concrete examples $H$ with open free orbits given in this section, it is possible to construct a larger group $H_1 \supset H$, with $H_1/H$ finite, and such that the action of $V \rtimes H_1^T$ has a \textit{single} open free orbit. Hence, the slightly larger translation complete subgroup $G = \mathbb{T} \times \mathbb{R}^n \times V \times H_1$ acts via a discrete series representation on all of $L^2(\mathbb{R}^n)$.
\end{remark}

\section*{Concluding remarks}
One could also conceive of the analogous class of \textit{modulation complete} groups, i.e., subgroups of $G_{aWH}$ of the type $G = \mathbb{T} \times W \times \widehat{\mathbb{R}}^n \times H$. However, rather than developing a collection of results analogous to the results of our paper, it is more efficient to observe that each such group is conjugate to a translation complete group via the Fourier transform, and all results related to admissible vectors -including admissible vectors, admissible subgroups, and related notions - can be readily transferred from one setting to the other.

Generalizing this observation, it could also be interesting to study conjugates of translation complete subgroups in the larger group $\mathcal{G} = \mathbb{H}^n \rtimes Sp(n,\mathbb{R})$, thereby allowing to transfer the sharp admissibility criteria that apply in the translation complete case to settings where currently the standard option for assessing admissible vectors is the Wigner transform, as e.g. in \cite{DeMari,Alberti} and related work. 

This paper focused on inversion formulae and admissibility criteria. The usefulness of such wavelet systems often hinges on their ability to efficiently approximate certain signal classes. These approximation-theoretic aspects of wavelet and time-frequency systems are best captured by their associated coorbit spaces, and for the affine case, these spaces are fairly well understood by now \cite{MR3345605,MR3378833}. We expect that these aspects carry over to translation complete subgroups as well, starting with the applicability of coorbit theory itself, which relies on the notion of \textit{integrable representation}. 

The reasoning developed for the affine case, in particular regarding integrable representations \cite{MR3378833,MR4361899}, the description of coorbit spaces as the so-called decomposition spaces \cite{MR3345605}, and the subsequent use made of this identification to classify dilation groups according to their associated coorbit spaces (e.g. as in \cite{MR4464548, MR4832531}), makes systematic use of the dual action to associate group parameters with frequency parameters and vice versa. It is highly plausible that this type of reasoning can be extended to translation complete groups as well.

\bibliographystyle{plain}
\bibliography{Ref.bib}





\end{document}